\documentclass[11pt,a4paper]{amsart}

\usepackage[utf8]{inputenc}
\usepackage[T1]{fontenc}
\usepackage{lmodern}
\usepackage{amsmath,amstext,amssymb,amsopn,amsthm,color,graphicx,enumitem,mathrsfs,mathtools,dsfont}
\usepackage[margin=1in]{geometry}
\usepackage{float}
\usepackage[colorlinks,citecolor=blue,urlcolor=blue,bookmarks=true]{hyperref}
\usepackage{tikz}
\usetikzlibrary{intersections}
\usetikzlibrary{arrows.meta}
\usetikzlibrary{decorations.text}
\usetikzlibrary{decorations.pathmorphing}
\usetikzlibrary{positioning}
\usetikzlibrary{calc}

 \hypersetup{
 	pdfpagemode=UseNone,
 	pdfstartview=FitH,
 	pdfdisplaydoctitle=true,
 	pdfborder={0 0 0}, 
 	pdftitle={Non-symmetric stable processes: Dirichlet heat kernel, Martin kernel and Yaglom limit},
 	pdfauthor={\L{}ukasz Le\.{z}aj},
 	pdflang=en-GB
 }
\usepackage{enumitem}
\usepackage{color}
\usepackage[colorinlistoftodos,prependcaption,textsize=tiny]{todonotes}

\newtheorem{theorem}{Theorem}[section]
\newtheorem{lemma}[theorem]{Lemma}
\newtheorem{proposition}[theorem]{Proposition}
\newtheorem{corollary}[theorem]{Corollary}

\theoremstyle{definition}

\newtheorem{example}[theorem]{Example}
\newtheorem{remark}[theorem]{Remark}

\newcommand{\C}{\mathbb{C}}

\newcommand{\E}{\mathbb{E}}
\renewcommand{\H}{\mathbb{H}}

\newcommand{\N}{\mathbb{N}}
\newcommand{\Sd}{\mathbb{S}^{d-1}}
\renewcommand{\P}{\mathbb{P}}
\newcommand{\R}{\mathbb{R}}
\newcommand{\Rd}{{\mathbb{R}^d}}

\newcommand{\bfX}{\mathbf{X}}

\newcommand{\calB}{\mathcal{B}}

\newcommand{\calM}{\mathcal{M}}

\DeclareMathOperator{\dist}{dist}
\DeclareMathOperator{\diam}{diam}

\newcommand{\ind}{\mathds{1}}

\newcommand{\one}{{\bf 1}}

\renewcommand{\subset}{\subseteq}

\newcommand{\vphi}{\varphi}
\renewcommand{\epsilon}{\varepsilon}

\renewcommand{\leq}{\leqslant}

\renewcommand{\geq}{\geqslant}

\newcommand{\Rdo}{\Rd \setminus \{0\}}
\newcommand{\hP}{\widehat{\P}}
\newcommand{\hbfX}{\widehat{\bfX}}
\newcommand{\ud}{\,{\rm d}}
\newcommand{\od}{{\rm d}}
\newcommand{\hM}{\widehat{M}}
\newcommand{\hbeta}{\widehat{\beta}}

\newcommand{\cbhi}{\mathrm{C_{BHI}}}

\definecolor{kb}{rgb}{0.1,0.5,0.1}

\numberwithin{equation}{section}
\allowdisplaybreaks

\begin{document}

\title[Non-symmetric stable processes: Dirichlet heat kernel and Yaglom limit]{Non-symmetric stable processes: Dirichlet heat kernel, Martin kernel and Yaglom limit}

\author{\L{}ukasz Le\.{z}aj}
\email{lukasz.lezaj@pwr.edu.pl}
\address{Address: Wroc\l{a}w University of Science and Technology, Department of Pure and Applied Mathematics, wyb. Wyspia\'{n}skiego 27, 50-370 Wroc\l{}aw, Poland}

\thanks{The author was partially supported by the National Science Centre (Poland): grant 2021/41/N/ST1/04139.}

\subjclass[2020]{Primary: 60J50, 60F05  secondary: 60G52, 60J35.}
 \keywords{non-symmetric stable process, heat kernel estimates, Dirichlet problem, Yaglom limit, Martin kernel}

\begin{abstract}
We study a $d$-dimensional non-symmetric strictly $\alpha$-stable L\'{e}vy process $\bfX$, whose spherical density is bounded and bounded away from the origin. First, we give sharp two-sided estimates on the transition density of $\bfX$ killed when leaving an arbitrary $\kappa$-fat set. We apply these results to get the existence of the Yaglom limit for arbitrary $\kappa$-fat cone. In the meantime we also obtain the spacial asymptotics of the survival probability at the vertex of the cone expressed by means of the Martin kernel for $\Gamma$ and its homogeneity exponent. Our results hold for the dual process $\hbfX$, too.
\end{abstract}

\maketitle

 %
 %

\section{Introduction}
The aim of this article is to establish certain results which are well known for rotation-invariant $\alpha$-stable L\'{e}vy processes, but have so far remained unknown in the non-symmetric case. Given $d \in \{1,2,\ldots\}$ and $\alpha \in (0,2)$, we let $\bfX$ be a strictly $\alpha$-stable L\'{e}vy process in $\Rd$. It is known that the L\'{e}vy measure $\nu$ of $\bfX$ has a following representation:
\begin{equation*}
	\nu(B) = \int_{\Sd} \,\zeta(\od\xi)\int_0^{\infty} \ind_{B}(r\xi)r^{-1-\alpha}\,\od r, \quad B \in \calB(\Rd).
\end{equation*}
The measure $\zeta$ is called the spherical measure and it describes the intensity of of jumps $\bfX$ in certain directions. We consider the following assumption:
\begin{enumerate}[topsep=0.3cm,leftmargin=*,labelindent=1cm,labelsep=1cm]
	\item[{\bf A}] The spherical measure $\zeta$ has a density $\lambda$ with respect to the surface measure, which is bounded and bounded away from zero, i.e. there is $\theta \in (0,1]$ such that
	\[
	\frac{\od\zeta}{\od\sigma} = \lambda \qquad \text{and} \qquad \theta \leq \lambda(\xi) \leq \theta^{-1}, \quad \xi \in \Sd.
	\]
\end{enumerate}
Informally speaking, assumption {\bf A} implies that the process $\bfX$ is, in a sense, \emph{similar} to the (classical) isotropic $\alpha$-stable L\'{e}vy process. Of course, if $\lambda$ is constant on the unit sphere, then one recovers the well-known rotation-invariant case with the fractional Laplacian $-(-\Delta)^{\alpha/2}$ as a generator. For an open set $D$ and $x,y \in D$, we let $p_t^D(x,y)$ be the Dirichlet heat kernel of $D$, that is the transition density of $\bfX$ killed when exiting $D$. Accordingly, for $x \in D$,
\[
\P_x(\tau_D>t) = \int_{D} p_t^D(x,y)\ud y
\]
is the survival probability of $\bfX$ in $D$. Let $\hbfX:=-\bfX$ be the dual process of $\bfX$. For consistency of the notation, all objects pertaining to $\hbfX$ will also be denoted by a superscript ,,\string^''. For instance, for $x,y \in \Rd$, we let $\widehat{p}_t(x,y) = p_t(y,x)$ be the (free) heat kernel of $\hbfX$. For the definition of $\kappa$-fat sets, see Section \ref{sec:prelims}. Our first result is the Varopoulos-type (see Varopoulos \cite{Varopoulos03}) factorisation of the Dirichlet heat kernel corresponding to the process $\bfX$ killed when exiting $D$.
\begin{theorem}\label{thm:dhk}
	Suppose that $R \in (0,\infty]$ and there is $\kappa \in (0,1)$ such that $D$ is $(\kappa,r)$-fat for all $r \in (0,R]$. Assume that $\bfX$ is a strictly $\alpha$-stable L\'{e}vy process satisfying {\bf A}. Then for all $c \geq 1$ there is $C = C(\alpha,d,\lambda,\kappa,c)$ such that for all $x,y \in D$ and $0 \leq t \leq cR^{\alpha}$,
	\begin{equation}\label{eq:main_thm_1}
	C^{-1}p_t^D(x,y) \leq \P_x(\tau_D>t)p_t(x,y)\hP_y(\tau_D>t) \leq Cp_t^D(x,y).
	\end{equation}
\end{theorem}
The proof of Theorem \ref{thm:dhk} is provided at the end of Section \ref{sec:dhk}. Put in context, it is an extension and natural continuation of Bogdan et al. \cite[Theorem 1]{KBTGMR10}, where the analogous result was proved for the rotation-invariant $\alpha$-stable L\'{e}vy process, i.e. when $\zeta$ is uniformly distributed on the unit sphere $\Sd$. Here we allow the strictly stable L\'{e}vy process to be non-symmetric (and in particular: anisotropic), as long as the distribution of the directions of jumps is absolutely continuous and its density is uniformly bounded and bounded away from zero. Observe that $\lambda$ need not be continuous and the domain $D$ may be unbounded. 

As a particular application, one can consider a $\kappa$-fat set $\Gamma$, which is invariant under rescaling, i.e. for every $r>0$ we have $ry \in \Gamma$, if only $y \in \Gamma$. To wit, $\Gamma$ is a generalised $\kappa$-fat cone in $\Rd$. By Theorem \ref{thm:dhk}, we immediately obtain global sharp two-sided estimate of the transition density of the Dirichlet heat kernel of $\Gamma$, i.e.
\[
p_t^{\Gamma}(x,y) \approx \P_x(\tau_{\Gamma}>t) p_t(x,y) \hP_y(\tau_{\Gamma}>t), \quad t>0,\,x,y \in \Gamma.
\]
Here the symbol $f\approx g$ means that the ratio of two functions $f$ and $g$ stays between two positive constants. In the second part ot the article we exploit the factorisation above to analyse the limiting behaviour of the process living in $\Gamma$. Denote $\one := (0,\ldots,0,1) \in \Rd$. Our second result gives the existence of the Yaglom limit for $\bfX$ in the $\kappa$-fat cone $\Gamma$.
\begin{theorem}\label{thm:Yaglom}
	Let $\bfX$ be a strictly $\alpha$-stable L\'{e}vy process satisfying {\bf A}. Assume $\Gamma$ is a $\kappa$-fat cone such that $\one \in \Gamma$. There is a probability measure $\mu$ concentrated on $\Gamma$ such that for every Borel set $A \subseteq \Gamma$ and every probability measure $\gamma$ on $\Gamma$ satisfying $\int_{\Gamma}(1+|y|)^{\alpha}\,\gamma(\od y)<\infty$,
	\begin{equation}\label{eq:Y}
	\lim_{t \to \infty} \P_{\gamma} \big( t^{-1/\alpha}X_t \in A \big| \tau_{\Gamma}>t \big) = \mu(A), \quad x \in \Gamma.
	\end{equation}
\end{theorem}
Theorem \ref{thm:Yaglom} is proved at the end of Section \ref{sec:Y} as Theorem \ref{thm:Y2}. In particular, by letting $\gamma$ be a Dirac delta at $x$, one can consider the process which starts from an arbitrary fixed point $x \in \Gamma$. Informally speaking, we show that, given the survival of $\bfX$ in the cone $\Gamma$, the limiting probability distribution of the properly rescaled process is independent of the initial distribution. In fact, the measure $\mu$ has a strictly positive density on $\Gamma$ given by the Martin kernel for the cone (see Lemma \ref{lem:M}) and the stationary density for the corresponding Ornstein-Uhlenbeck operators from Theorem \ref{thm:s_density} (for the exact formula for the density, see Section \ref{sec:Y} and Theorem \ref{thm:Y2}). 

Let us comment on the methodology and refer to previous developments in the literature. Theorem \ref{thm:dhk} is a next member of a large family of results concerning (Dirichlet) heat kernel estimates of stable processes in various domains. The analysis began with the estimates of the Green function in $C^{1,1}$ domains for the (classical) Laplace operator given by Zhao \cite{Zhao86}. The extension to Lipschitz domains was accomplished by Bogdan \cite{KB00}. Shortly afterwards, Zhang \cite{Zhang02} established sharp two-sided estimates of the Dirichlet heat kernel in $C^{1,1}$ domains, which were extended later to Lipschitz domains in Varopoulos \cite{Varopoulos03}. A parallel track was taken for the (non-local) fractional Laplacian, where the Green functions estimates for $C^{1,1}$ domains were obtained by Kulczycki \cite{Kulczycki97} and Chen and Song \cite{MR1654824}, and for Lipschitz domains --- by Jakubowski \cite{MR1991120}. The last steps were provided by Chen et al. \cite{MR2677618} and Bogdan et al. \cite{KBTGMR10}, who developed sharp two sided estimates of the Dirichlet heat kernel in $C^{1,1}$ and Lipschitz domains, respectively. The paper \cite{KBTGMR10} is our roadmap in the first part of the article.

We note that since then, analogous results has been obtained for various types of symmetric processes, see, for instance, Chen et al. \cite{MR2981852, MR2917772}. In a more general setting, Bogdan et al. \cite{KBTGMR-dhk} and Chen et al. \cite{CKS14} considered a large class of unimodal L\'{e}vy process whose characteristic exponent satisfies weak scaling conditions, in particular when the upper scaling index is strictly smaller than $2$. Subordinated Brownian motions which do not satisfy this assumption were studied by Kim and Mimica \cite{MR3835470}. Grzywny et al. \cite{MR4035035} considered symmetric Markov processes whose jumping kernel is comparable to the one of an unimodal L\'{e}vy process. Gradient estimates of the Dirichlet heat kernel were studied by Kulczycki and Ryznar \cite{MR3413864}. Out of vast amount of results, we would also like to mention a recent preprint of Chen et al. \cite{un_Chen-Hu-Zhao23}, who studied the Dirichlet heat kernel for cylindrical $\alpha$-stable L\'{e}vy process, a case --- in a sense --- complementary to our situation. For more information we refer the reader to the articles above and the references therein.

As already mentioned, the techniques used to derive the factorisation of the Dirichlet heat kernel are inspired by \cite{KBTGMR10}, with the necessary changes concerning the non-symmetry of the process. To give one example, the Dirichlet heat kernel is not symmetric anymore, which results in the necessity to work simultaneously with $\bfX$ and its dual $\hbfX$. Accordingly, the factorisation \eqref{eq:main_thm_1} is also expressed by means of objects pertaining to both the initial process $\bfX$ and its dual $\hbfX$, c.f. \cite[Eq. (0.25)]{Varopoulos03}. For this reason, the boundary behaviour of the Dirichlet heat kernel may be significantly different in variable $x$ than in $y$. In fact, this is the case even in the simple case of the half-space, as we advocate in Example \ref{ex:halfspace2}. It is crucial in our development that despite the non-symmetry of the process, our assumptions on $\bfX$ \emph{are} symmetric in $\bfX$ and $\hbfX$. In particular, the boundary Harnack principle holds for both $\bfX$ and $\hbfX$ (with the same constants) by virtue of Bogdan et al. \cite{BKK15}. However, other classical bricks in the potential theory of $\bfX$ such as the explicit formula for the Green function or the Poisson kernel of the ball are not available in the general non-symmetric case.
Nonetheless, the ideas of \cite{KBTGMR10} can be transferred to our setting. Here we should note that some of basic tools from the potential theory are provided by the accordingly titled paper by Vondra\^cek \cite{Vondracek02} and we frequently refer to his work in what follows. Interestingly, the \emph{standard} track of results (that is: Dirichlet heat kernel estimates follow sharp Green function estimates) seems to be disturbed in our case.

The application of Theorem \ref{thm:dhk} to unbounded cones provides the crucial tool in the analysis of the limiting distribution of the process conditioned not to leave the cone. This limit, if exists, is called the Yaglom limit due to the seminal paper of Yaglom \cite{Yaglom47}, who identified the so-called quasi-stationary distribution for the Bienaym\'{e}-Galton-Watson trees. Since then, it has been studied in various settings and for different types of stochastic processes and we refer to Van Doorn and Pollet \cite{MR3063313} for the comprehensive study. Out of vast amount of papers, let us mention Seneta and Vere-Jones \cite{MR0207047} and Tweedie \cite{Tweedie74} in the discrete-time Markov chain, and Jacka and Roberts \cite{JackaRoberts95} in the continuous-time case. Markov chains on non-negative integers with the origin as the absorbing state was studied i.a. by Ferrari et al. \cite{MR1334159}, van Doorn \cite{MR1133722} and Flaspohler and Holmes \cite{MR0346932}. See also e.g. Bean et al. \cite{BBLPPT} for quasi-birth-and-death processes, Asselah et al. \cite{MR3498004} for Fleming-Viot process, Lambert \cite{MR2299923} for branching processes or the recent paper by Harris et al. \cite{MR4499840} for the non-local branching Markov processes. This short list is far from complete and for more information we refer the reader to the references in the mentioned articles and to the survey \cite{MR3063313}.

The study of L\'{e}vy processes in the context of Yaglom limit and quasi-stationary distributions started with Mart\'{\i}nez and San Mart\'{\i}n \cite{SMJS94}, who studied the case of Brownian motion with drift as a counterpart of results obtained by Iglehart \cite{MR0368168} for the random walk with negative drift. Spectrally one-sided L\'{e}vy processes were investigated by Kyprianou and Palmowski \cite{MR2248228}, Mandles et al. \cite{MR2959448} or Palmowski and Vlasiou \cite{MR4171931}. One-dimensional self-similar Markov processes were considered by Hass and Rivero \cite{HassRivero12} and we note that in this case the appropriate rescaling of the process similar to \eqref{eq:Y} is essential to obtain a non-trivial limit. In the multi-dimensional case, Bogdan et al. \cite{BoPaWa2018} obtained the Yaglom limit in Lipschitz (generalised) cones for the isotropic $\alpha$-stable L\'{e}vy process. Recently Armstrong et al. \cite{ABGLW23-aihp} showed that the limit is in fact the same for every unimodal L\'{e}vy process which is in domain of attraction of the isotropic $\alpha$-stable law. We should also mention results of Zhang et al. \cite{MR3247530} who studied the Yaglom limit of Markov process killed when leaving sets of bounded volume, but their approach does not apply to general unbounded cones. We note in passing that the problem we study is intrinsically connected to, but nonetheless different than the conditioning of the process to stay forever in a certain set. For more information we refer the reader to Bertoin and Doney \cite{MR1331218}, Chaumont \cite{MR1419491} and Chaumont and Doney \cite{MR2164035}, or to Kyprianou et al. \cite{MR4415390}, where the isotropic $\alpha$-stable L\'{e}vy process conditioned to stay in a Lipschitz cone is considered.
 
The article \cite{BoPaWa2018} and its successor \cite{ABGLW23-aihp} were both based on a tricky compactness argument and a formula expressing the survival probability as a Green potential, which, after a rather technical argument, yields the spacial asymptotics of the survival probability at the vertex of the cone by means of its Martin kernel. A completely different approach was applied in the recent preprint of Bogdan et al. \cite{un_BKLP23}, where the authors considered the generalised Ornstein-Uhlenbeck semigroup and established the existence of its stationary density, which was later used to derive the existence of the Yaglom limit for more general family of $\kappa$-fat cones. This line of attack seems more versatile, since once the existence of the aforementioned stationary density is proved, the remaining part follows by the scaling property of the process and therefore, it should be applicable for other self-similar processes which enjoy both sharp two-sided estimates on the Dirichlet heat kernel and a $P_t^{\Gamma}$-invariant function which allows for Doob-type conditioning. In \cite{un_BKLP23}, this role is played by the Martin kernel for $\Gamma$ with the pole at infinity, whose existence was established by Ba\~{n}uelos and Bogdan \cite{BB04}. We follow this path and collect first some results and methods from the literature to conclude the existence and basic properties of the Martin kernel for the non-symmetric $\alpha$-stable process $\bfX$ and its dual $\hbfX$ in Lemma \ref{lem:M}. The crucial property is the homogeneity of order $\beta \in [0,\alpha)$, which, perhaps unsurprisingly, turns out to be different for $\bfX$ and $\hbfX$ even when $\Gamma$ is a half-space, see Example \ref{ex:halfspace}. For this reason, the rest of analysis performed in Section \ref{sec:Y} needs to be more delicate. For instance, the definition of the conditioned kernel $\rho_t$ \eqref{eq:rho_def} involves Martin kernels for both the starting process and its dual counterpart. This is due to the structure of factorisation in Theorem \ref{thm:dhk}, which implies that the sharp estimate of $\rho_t$ in \eqref{eq:rho_f} is, in a sense, \emph{symmetric} in $\bfX$ and $\hbfX$. Proposition \ref{prop:M_surv} then provides the appropriate control of the first and the last term of \eqref{eq:rho_f}. With these tools at our disposal, one can obtain the stationary density $\vphi$ of the corresponding Ornstein-Uhlenbeck semigroup $L_t$ (see \eqref{eq:26} and \eqref{eq:Lt} for definitions and Theorem \ref{thm:s_density} for the statement) and use it to conclude the existence of the desired Yaglom limit. Similarly to \cite{BoPaWa2018}, we first obtain the spacial asymptotics of the survival probability in Corollary \ref{cor:surv_M_as} and then use it for the proof of Theorem \ref{thm:Yaglom}. Observe that the Yaglom measure $\mu$ has a density which is expressed by objects pertaining to both $\bfX$ and $\hbfX$. Clearly, if $\bfX$ is symmetric, then $\bfX \stackrel{{\rm d}}{=} \hbfX$, but even in this case we extend results from \cite{BoPaWa2018} and \cite{un_BKLP23}, where the isotropic case was considered. Thus, Theorem \ref{thm:Yaglom} may be perceived as the next stage of development concerning the analysis of self-similar processes living in scale-invariant open sets.

Finally, for the sake of faithfulness we note that the roadmap for the method applied in \cite{un_BKLP23}, which we extend in this work, was set by Bogdan et al. \cite{BJKP23_jfa}, where the Hardy operator on $\Rd \setminus \{0\}$ was considered.

The article is organised as follows. Section \ref{sec:prelims} contains basic preliminary results concerning the non-symmetric strictly $\alpha$-stable processes, as well as the crucial bound on the growth of the harmonic functions close to the boundary (see Lemma \ref{lem:harmonic_growth_bound}). Section \ref{sec:dhk} is devoted to the proof of Theorem \ref{thm:dhk}. In Section \ref{sec:Y} we first discuss the Martin kernel of the cone $\Gamma$ and provide a simple, but relevant example, where the exponents $\beta$ and $\hbeta$ differ (see Example \ref{ex:halfspace}). This case also yields an example for Theorem \ref{thm:dhk}, where the phenomenon of different boundary behaviour in variables $x$ and $y$ is visible, see Example \ref{ex:halfspace2}. Then we apply the derived tools to the analysis of the corresponding generalised Ornstein-Uhlenbeck semigroup (Theorem \ref{thm:s_density}) and use it to prove Theorem \ref{thm:Yaglom}. 
\subsection*{Notation}
Throughout the article, $c$ will be a positive constant, which may vary from line to line in chain of estimates. Sometimes we will use the notation $c_1,c_2,\ldots$ to distinguish certain constants. By writing $c=c(a)$ we mean that $c$ depends at most on the parameter $a$. The notation $c=c(a,b,\ldots)$ is defined accordingly. For $x,z \in \Rd$, we denote by $x \cdot z$ the standard scalar product in $\Rd$. As usual, $|x|$ is the Euclidean norm. For $r>0$ we let $B(x,r):=\{z \in \Rd \colon |x-z|<r\}$ be the ball centred in $x$ of radius $r$. To abbreviate the notation, we will write $B_r:=B(0,r)$. For an open set $D$, we denote by $\delta_{D}(x) := \inf \{ |y-x|\colon y \in D^c \}$ the distance from the complement of $D$. Similarly, $\dist(x,D):=\inf\{ |y-x| \colon y \in D \}$ and $\dist(D_1,D_2) := \inf_{x \in D_1} \dist(x,D_2)$ for open sets $D_1$ and $D_2$. All considered sets, functions and measures are assumed to be Borel. As already noted, for two positive functions $f$ and $g$, by writing $f \approx g$ we mean that the ratio of $f$ and $g$ is bounded from above and below by positive constants. The symbols $\lesssim$ and $\gtrsim$ are defined accordingly. The set of $d \times d$ matrices with real entries will be denoted by $\calM_d$. To avoid unnecessary considerations, we set $\N = \{1,2,\ldots\}$. Finally, we let $a \wedge b = \min \{a,b\}$ and $a \vee b = \max \{a,b\}$.

\section{Preliminaries}\label{sec:prelims}
Let $d \in \N$. Throughout the paper we assume that $\bfX$ is a $d$-dimensional strictly stable L\'{e}vy process with the stability index $\alpha \in (0,2)$, that is for all $a>0$ the following equality of distributions holds:
\begin{equation}\label{eq:33}
\big(X_{at}: t \geq 0\big) \stackrel{d}{=} \big(a^{1/\alpha}X_t : t \geq 0\big).
\end{equation}
We note that for $\alpha=2$ one recovers Brownian motion, the case excluded from our considerations. Thus, there is a function $\psi \colon \Rd \mapsto \C$ such that
\[
\E e^{i \xi \cdot X_t} = e^{-t\psi(\xi)}, \quad \xi \in \Rd,
\]
and
\[
\psi(\xi) = -i \xi \cdot \gamma - \int_{\Rd} \big( e^{i\xi \cdot z}-1-i \xi \cdot z \ind_{|z|<1} \big) \,\nu(\od z), 
\]
where $\gamma \in \Rd$ is the drift component and $\nu$ is the L\'{e}vy measure satisfying
\[
\nu(\{0\})=0 \qquad \text{and} \qquad \int_{\Rd} \big(1 \wedge |z|^2\big)\,\nu(\od z)<\infty.
\]
Moreover, there is a finite measure $\zeta$ on $\Sd$ such that
\begin{equation}\label{eq:nu_spherical}
	\nu(B) = \int_{\Sd} \,\zeta(\od\xi)\int_0^{\infty} \ind_{B}(r\xi)r^{-1-\alpha}\,\od r, \quad B \in \calB(\Rd),
\end{equation}
see e.g. \cite[Theorem 14.3]{Sato99}. The measure $\zeta$ is called the spherical measure and it describes the non-isotropic intensity of the expansion of the process $\bfX$ in different directions. We note here that the trivial case $\alpha=1$ and $\nu \equiv 0$, that is $\bfX$ being a deterministic drift, is excluded. Therefore, the following characterisation of strictly stable L\'{e}vy processes in our setting holds true (see e.g. \cite[Theorem 14.7]{Sato99}):
\begin{itemize}
	\item Let $\alpha \in (0,1)$. $\bfX$ is strictly stable if and only if
	\begin{equation*}
	\gamma -\int_{B_1} z \,\nu(\od z) = 0.
	\end{equation*}
	\item Let $\alpha = 1$. $\bfX$ is strictly stable if and only if
	\begin{equation*}
	\gamma = 0 \qquad \text{and} \qquad \int_{\Sd} \xi \,\zeta(\od\xi) = 0.
	\end{equation*}
	\item Let $\alpha \in (1,2)$. $\bfX$ is strictly stable if and only if
	\begin{equation*}
	\gamma + \int_{B_1^c} z \,\nu(\od z) = 0.
	\end{equation*}
\end{itemize}
Put differently, every finite measure on $\Sd$ induces a strictly $\alpha$-stable L\'{e}vy process. The conditions above imply that, given $\alpha \in (0,2)$, the \emph{average expansion rate} of every strictly $\alpha$-stable L\'{e}vy process is the same. To make this statement precise, following Pruitt \cite{Pruitt81}, we define
\begin{align*}
	h(r) 
	= r^{-2}\int_{B_r}|z|^2\,\nu(\od z) + \nu(B_r^c) + r^{-1} \bigg|\gamma + \int_{\Rd} z \big( \ind_{|z|<r}-\ind_{|z|<1} \big) \,\nu(\od z)\bigg|, \quad r > 0.
\end{align*}

Using \eqref{eq:nu_spherical} together with the strict stability, it is easy to show that there is $c = c(\alpha,d,\zeta)$ such that
\begin{equation}\label{eq:h}
	h(r) = cr^{-\alpha}, \quad r>0.
\end{equation}
If we define $S(r) = \inf \{t>0 \colon |X_t|>r\}$, then \cite{Pruitt81} yields that $\E S(r) \approx h(r)$, which, by the equation above, justifies the \emph{similar average expansion} of strictly stable processes.

It is known that $\bfX$ is a Markov process with transition function given by
\[
P_t(x,A) = \int_A p_t(x,y)\ud y,
\]
where $p_t(x,y):=p_t(y-x)$ is the probability density function, so that
\[
\int_{\Rd} p_t(x)e^{i\xi\cdot x}\ud x = e^{-t\psi(\xi)}, \quad \xi \in \Rd.
\]
By Hartman and Wintner \cite{HW42}, $p_t$ is smooth and integrable for all $t>0$. The strict stability \eqref{eq:33} translates into the scaling property of the transition density as follows:
\begin{equation}\label{eq:hk_scaling}
	p_t(x,y) = t^{-d/\alpha}p_1 \big( t^{-1/\alpha}x,t^{-1/\alpha}y \big), \quad x,y \in \Rd, \, t>0.
\end{equation}
In particular, $p$ is jointly continuous on $(0,\infty) \times \Rd \times \Rd$.

Our {\bf standing} assumption in this article is that the spherical measure $\zeta$ has a density $\lambda$ with respect to the surface measure $\sigma$ which is bounded and bounded away from zero, that is
\begin{equation}\label{eq:main_as}
\frac{\od\zeta}{\od\sigma} = \lambda \qquad \text{and} \qquad \theta \leq \lambda(\xi) \leq \theta^{-1}, \quad \xi \in \Sd, 
\end{equation}
for some $\theta \in (0,1)$. It is clear that if $\lambda$ is constant, then $\zeta$ is the uniform distribution on the sphere $\Sd$ and we recover the L\'{e}vy measure of the isotropic $\alpha$-stable L\'{e}vy process. One immediately obtains from \eqref{eq:main_as} that the L\'{e}vy measure $\nu$ of $\bfX$ is absolutely continuous with respect to the Lebesgue measure on $\Rd$ with the density satisfying $\nu(x) = \lambda(x/|x|)|x|^{-d-\alpha}$, $x \in \Rd \setminus \{0\}$, and
\begin{equation}\label{eq:nu_approx}
\theta |x|^{-d-\alpha} \leq \nu(x) \leq \theta^{-1}|x|^{-d-\alpha}, \quad x \in \Rdo.
\end{equation}
By \cite[Theorem 1.5]{Watanabe07} and the scaling property \eqref{eq:hk_scaling} (see also \cite{Vondracek02} for some partial results),
\begin{equation}\label{eq:hk_approx}
	p_t(x,y) \approx t^{-d/\alpha} \wedge t|x-y|^{-d-\alpha}, \quad x,y \in \Rd,
\end{equation}
and in particular,
\begin{equation}\label{eq:11}
	p_1(x,y) \approx (1+|y-x|)^{-d-\alpha}, \quad x,y \in \Rd.
\end{equation}
Accordingly, $p_1(0,0)>0$ and $\bfX$ is of type A in the terminology of Taylor \cite{Taylor67}. Combining the equation above with \eqref{eq:nu_approx}, we obtain
\begin{equation}\label{eq:10}
	p_1(x,y) \approx 1 \wedge \nu(y-x) \approx p_1(y,x), \quad x,y \in \Rd.
\end{equation}
It follows also that for any constant $c >0$ one has
\begin{equation}\label{eq:12}
	p_t(x,y) \approx p_{ct}(x,y),
\end{equation}
with the implied constant dependent on $c$.
\subsection{Killed process}
Assume $D$ is a domain in $\Rd$. We let $\tau_D$ be the fist exit time from $D$, i.e.
\[
\tau_D = \inf \{t \geq 0 \colon X_t \notin D\}.
\]
The process $\bfX^D$ killed after exiting $D$ is then defined by
\begin{align*}
	X_t^D = \begin{cases}
		X_t, & t < \tau_D,\\
		\partial, & t \geq \tau_D,
	\end{cases}
\end{align*}
where $\partial$ is a cemetery point adjoined to $D$. For every $t \geq 0$ and $x,y \in D$ we define the (Dirichlet) heat kernel
\begin{equation*}
	p_t^D(x,y) = p_t(x,y) - \E_x [\tau_D<t;p_{t-\tau_D}(X_{\tau_D},y)].
\end{equation*}
Then it follows that $p_t^D(x,y) \leq p_t(x,y)$ and consequently,
\begin{equation}\label{eq:3}
	\int_D p_t^D(x,y)\ud y \leq 1.
\end{equation}
In the analogous way we define the dual killed process $\hbfX^D$ and the dual (Dirichlet) heat kernel $\widehat{p}_t^D$. The function $(t,x,y) \mapsto p_t^D(x,y)$ is continuous on $(0,\infty) \times D \times D$ and satisfies the Chapman-Kolmogorow property: for every $t,s>0$ and $x,y \in D$,
\begin{equation}\label{eq:dhk_CH-K}
	p_{t+s}^D(x,y) = \int_D p_t^D(x,z)p_s^D(z,y)\ud z.
\end{equation}
Moreover, for all $x,y \in D$ and all $t>0$,
\begin{equation}\label{eq:dual_dhk}
	p_t^D(x,y) = \widehat{p}_t^D(y,x).
\end{equation}
 
These (basic) properties 
are stated and proved in \cite{Vondracek02} as part of Theorem 3.2. Its proof is basically the application of arguments from Chung and Zhao \cite[Section 2.2]{ChungZhao95} for the (classical) Brownian motion case and Chen et al. \cite{MR1473631} with the necessary bounds on the heat kernel $p_t$ \cite[Proposition 2.1]{Vondracek02} and in fact, they hold for arbitrary domains $D \subseteq \Rd$. The analogous property for isotropic $\alpha$-stable process is well known, see e.g. \cite[Theorems 2.3 and 2.4]{MR1473631} and also follows from \cite[Section 2.2]{ChungZhao95}. For every $t>0$, $x \in D$ and $f \in L^{\infty}(D)$,
\[
P_t^D f(x):= \E_x[t<\tau_D;f(X_t)] = \int_D p_t^D(x,y)f(y)\ud y.
\]
To wit, $P_t^D$ is a killed semigroup corresponding to the killed strictly stable process $\bfX^D$. In particular, setting $f \equiv 1$ yields the survival probability
\begin{equation}\label{eq:surv_prob}
\P_x(\tau_D>t) = \int_D p_t^D(x,y)\,\od y.
\end{equation}
It goes without saying that the same results hold for the dual process $\hbfX^D$, too. By \eqref{eq:hk_scaling}, the scaling property holds for the Dirichlet heat kernel, too:
\begin{equation}\label{eq:dhk_scaling}
	p_t^D(x,y) = t^{-d/\alpha}p_1^{t^{-1/\alpha}D}\big(t^{-1/\alpha}x,t^{-1/\alpha}y\big), \quad x,y \in D, \,t>0.
\end{equation}
Put differently,
\begin{equation}\label{eq:dhk_scaling2}
	p_{r^{\alpha}t}^{rD}(rx,ry) = r^{-d}p_t^D(x,y), \quad r>0,\,x,y \in D.
\end{equation}
Therefore, one also has
\begin{equation}\label{eq:surv_scaling}
	\P_x(\tau_D>t) = \P_{t^{-1/\alpha}x}\Big( \tau_{t^{-1/\alpha}D}>1 \Big), \quad x \in \Rd,\,t>0.
\end{equation}

Throughout the article, we will often impose some restrictions on the geometry of the open set $D$. We will say that $D$ is $(\kappa,r)$-fat at $Q \in \overline{D}$, if there is a point $A=A_r(Q) \in D \cap B(Q,r)$ such that $B(A,\kappa r) \subseteq D \cap B(Q,r)$. If this property holds for all $Q \in \overline{D}$, then $D$ is said to be $(\kappa,r)$-fat. If there is $R>0$ such that $D$ is $(\kappa,r)$-fat for all $r \in (0,R]$, then we simply say that $D$ is $\kappa$-fat. For example, every Lipschitz set is $\kappa$-fat. Of course, if $D$ is $\kappa$-fat, then it is also $\kappa'$-fat for all $\kappa' < \kappa$. We also note that in general, the point $A_r(Q)$ need not be uniquely determined, but we choose to accept for this slight abuse of notation.

We say that a function $u \colon \Rd \mapsto \R$ is harmonic (with respect to $\bfX$) in an open set $D \subseteq \Rd$ if for every bounded open set $B \subseteq D$,
\[
u(x) = \E_x \big[\tau_B<\infty;u(X_{\tau_B})\big], \quad x \in B,
\]
where the integral on the right-hand side is assumed to be absolutely convergent. If the identity above is satisfied with $B=D$, then $u$ is regular harmonic in $D$. Also, we will say that $u$ is (regular) co-harmonic, if it is (regular) harmonic with respect to the dual process $\hbfX$. The probability distribution of the $X_{\tau_D}$ is the harmonic measure $P_D(x,\,\cdot\,)$, that is
\[
\P_x(X_{\tau_D} \in A) = \int_A \, P_D(x,\od y), \quad A \subset \Rd.
\]
Thus, a regular harmonic function $u$ satisfies
\[
u(x) = \int_{D^c} u(y)\,P_D(x,\od y), \quad x \in D.
\]
The Green function $G_D(x,y)$ is given by
\[
G_D(x,y) = \int_0^\infty p_t^D(x,y)\ud t, \quad x,y \in D.
\]
We will say that $D$ is Greenian if $G_D$ is finite almost everywhere on $D \times D$. This is the case, if, for example, $D$ is bounded. If $\alpha<d$, then $\bfX$ is transient and every $D \subseteq \Rd$ is Greenian, see \cite[Chapter 7]{Sato99}. For $d=1$ and $\alpha \in (1,2)$ and every $D\neq\R$ is Greenian e.g. by Port \cite{Port67}. If $\alpha=d=1$, then $\bfX$ is a symmetric Cauchy process and the explicit formula for the Green function of the half-line was given by M. Riesz, see, for example, Blumenthal et al. \cite{MR0126885}. On the other hand, in this case singletons are polar sets by \cite[Remark 43.6 and Example 43.7]{Sato99} and consequently, neither $\R$ nor $\R \setminus \{0\}$ are Greenian. For details we refer the reader to \cite[Chapter 8]{Sato99}.

Throughout the paper we assume that $G_D(x,y)=0$ whenever $x \in D^c$ or $y \in D^c$. The scaling property of $p^D$ \eqref{eq:dhk_scaling2} implies the scaling of $G_D$ in the following form:
\begin{equation}\label{eq:GD_scaling}
G_{rD}(rx,xy) = r^{\alpha-d}G_D(x,y), \quad r>0,\,x,y \in D.
\end{equation}
Moreover, observe that by the strong Markov property, for every $B \subset D$,
\[
p_t^D(x,y) = p_t^B(x,y) + \E_x [\tau_B<t;p_{t-\tau_B}^D(X_{\tau_B},y)].
\]
Integrating this identity with respect to $\od t$ yields
\[
G_D(x,y) = G_B(x,y) + \E_x G_D(X_{\tau_B},y), \quad x,y \in D.
\]
Therefore, $x \mapsto G_D(x,y)$ is regular harmonic in $B$ (and consequently: continuous in $B$, c.f. \cite[Theorem 6.7]{Vondracek02}), provided that $y \in D \setminus \overline{B}$. One also concludes that $G_B(x,y) \leq G_D(x,y)$ for all $x,y \in \Rd$. We note in passing that if $\alpha<d$, then (finite) $U:=G_{\Rd}$ is called the potential kernel of $\bfX$. It follows from the scaling property \eqref{eq:hk_scaling} together with \eqref{eq:11} that
\begin{equation*}
U(x,y) \approx |x-y|^{\alpha-d}, \quad x \neq y,
\end{equation*}
i.e. the potential kernel of $\bfX$ is comparable to the one of the isotropic $\alpha$-stable L\'{e}vy process. The analogous property for the Green function of an arbitrary open set holds away from the boundary of $D$, see e.g. \cite[Theorem 4.4]{Vondracek02}. For more information we refer the reader to \cite{Vondracek02}.

Let us comment on the harmonic measure $P_D$. The celebrated Ikeda-Watanabe formula \cite{IW62} states that for $x \in D$, the joint distribution of $(\tau_D,X_{\tau_D-},X_{\tau_D})$ restricted to the event $\{\tau_D<\infty, X_{\tau_D-} \neq X_{\tau_D}\}$ has a density function given by
\begin{equation}\label{eq:IW_density}
(0,\infty) \times D \times \overline{D}^c \ni (s,u,z) \mapsto p^D_s(x,u)\nu(z-u).
\end{equation}
In particular, the distribution of $X_{\tau_D}$ on $\overline{D}^c$ has a density function called the Poisson kernel of $D$:
\begin{equation}\label{eq:IW}
P_D(x,z) = \int_D G_D(x,y)\nu(z-y)\,\od y, \quad x \in D, z \in \overline{D}^c.
\end{equation}
Therefore, under the condition that the process does not hit the boundary when exiting the set $D$, i.e. when $\P_x(X_{\tau_D} \in \partial D)=0$, by the discussion above one has
\begin{equation*}
\P_x(\tau_D \in U) = \int_U P_D(x,z)\,\od z, \quad U \subseteq \Rd.
\end{equation*}
This is the case if $D$ is a Lipschitz domain, see Sztonyk \cite[Theorem 1.1 and the discussion below it]{Sztonyk00} and Millar \cite{Millar75}.

To simplify the notation, we denote the Poisson kernel of the ball $B_r$ by $P_r$. The following immediate extension of \cite[Lemma 5.3]{Vondracek02} will be useful in what follows.
\begin{proposition}\label{prop:Pr_est}
	Let $r>0$ and $\epsilon>0$. For all $y \in B_r$ and $z \in B_{(1+\epsilon)r}^c$ we have
	\[
	\theta \bigg( \frac{\epsilon}{1+\epsilon}\bigg)^{d+\alpha} \frac{\E_y \tau_{B_r}}{|z|^{d+\alpha}} \leq P_r(y,z) \leq \theta^{-1} \bigg( \frac{2+\epsilon}{1+\epsilon} \bigg)^{d+\alpha} \frac{\E_y \tau_{B_r}}{|z|^{d+\alpha}}.
	\]
\end{proposition}
In particular, it follows by Pruitt's estimates \cite{Pruitt81} and \eqref{eq:h} that $P_r(y,z)$ is strictly positive for all $y \in B_r$ and $z \in \overline{B}_r^c$.

\subsection{Properties of harmonic functions}
Let us recall two basic tools of the potential theory. First, by \cite[Theorem 6.7 and Corollary 6.9]{Vondracek02} (see also Song and Vondra\^{c}ek \cite{SV04}), for any connected open set $D$ and any compact subset $K \subseteq D$ there is a constant $c = c(d,\alpha,\lambda,D,K)$ such that for every function $h$ which is non-negative in $\Rd$ and harmonic in $D$,
\begin{equation}\label{eq:HI}\tag{HI}
	h(x) \leq c h(y), \quad x,y \in K.
\end{equation}
The property \eqref{eq:HI} is the \emph{Harnack inequality} and it describes the behaviour of harmonic functions inside the domain of harmonicity. Note that for the symmetric case, this follows also from a seminal paper of Bass and Levin \cite{BassLevin02}. The phenomena close to the boundary are governed by the \emph{boundary Harnack principle}, which we now invoke. Let $D$ be an arbitrary open set. In view of Bogdan et al. \cite[Example 5.5]{BKK15}, the global scale-invariant boundary Harnack inequality holds for $\bfX$, i.e. for every $r>0$ and all non-negative functions $f,g$ which are regular harmonic in $D \cap B(x_0,2r)$ and vanish on $D^c \cap B(x_0,2r)$ one has
\begin{equation}\label{eq:BHP}\tag{BHP}
	\frac{f(x)}{g(x)} \leq \cbhi \frac{f(y)}{g(y)}, \quad x,y \in D \cap B(x_0,r),
\end{equation}
where $\cbhi =\cbhi(\alpha,d,\lambda)$. Note that all the assumptions are symmetric for $\bfX$ and $\hbfX$ and consequently, \eqref{eq:HI} and \eqref{eq:BHP} hold for regular co-harmonic functions, too.

The following lemma gives the upper bound on the growth of harmonic functions close to the boundary and is crucial for our development. It is a generalisation of \cite[Lemma 5]{KB97}, where the isotropic case was considered. We point out that, as in the rotation-invariant setting, the growth order $\gamma$ is strictly less than $\alpha$. 
\begin{lemma}\label{lem:harmonic_growth_bound}
	Let $R \in (0,\infty]$ and assume that an open set $D$ is $(\kappa,r)$-fat for some $\kappa \in (0,1)$ and all $r \in (0,R)$. There are $c = c(\alpha,d,\lambda,\kappa,R) \in (0,\infty)$ and $\gamma \in [0,\alpha)$ such that for all boundary points $x_0 \in \partial D$, all $r \in (0,R)$ and all non-negative functions $f$ which are regular harmonic in $D \cap B(x_0,r)$,
	\[
	f(A_s(x_0)) \geq c \bigg(\frac{|A_s(x_0)-x_0|}{r}\bigg)^\gamma f(A_r(x_0)), \quad s \in (0,r).
	\]
\end{lemma}
\begin{proof}
	We follow the proof of \cite[Lemma 5]{KB97}. In fact, the only difference is the derivation of the inequality above \cite[eq. (3.32)]{KB97}. Once this relation is proved, the proof will follow by exactly the same arguments. 
	
	We apply the analogous notation and geometry. First, without loss of generality we may assume that $x_0=0$, $r=1$ and $f(A_r(0))=1$. Let $T = 2/\kappa$. For $k=0,1,\ldots$, we set $r_k = T^{-k}$, $A_k = A_{r_k}(0)$ and $D_k = B(A_k,r_{k+1})$. We then clearly have for $k>l \geq 0$ and $y \in D_l$ that $|D_l| \geq c_1 T^{-dl}$, $\diam(D_k) \geq c_1 T^{-k}$ and $\dist(A_k,y) \leq c_1 T^{-l}$ for some constant $c_1=c_1(d,\kappa)$. Next, by the Harnack inequality \eqref{eq:HI}, there is $c_2=c_2(d,\alpha,\lambda,\kappa,R)$ such that
	\[
	\frac{f(x)}{f(y)} \leq c_2, \quad x,y \in D_k, \, k=0,1,\ldots.
	\]
	It follows that
	\begin{align*}
		f(A_k) &= \int_{D_k^c} f(z) P_{D_k}(A_k,z)\,\od z \geq c_2^{-1} \sum_{l=0}^{k-1} f(A_l) \int_{D_l}P_{r_{k+1}}(0,z-A_k)\,\od z.
	\end{align*}
	Observe that for $z \in D_l$, $l=0,1,\ldots,k-1$, one also has $|z-A_k|\geq \kappa r_k$. Therefore, using Proposition \ref{prop:Pr_est} with $\epsilon = (\kappa r_k)/r_{k+1}-1=1$, together with \eqref{eq:h}, we conclude that
	\[
	f(A_k) \geq c_3 \sum_{l=0}^{k-1}T^{-(k-l)}f(A_l),
	\]
	where $c_3 = c_3(d,\alpha,\lambda,\kappa)$. With the inequality above at hand, one can directly copy the proof of \cite[Lemma 5]{KB97} to conclude the claim, where in place of \cite[Lemma 2]{KB97} one may use \cite[Lemma 6.8]{Vondracek02} or simply apply the (classical) Harnack inequality \eqref{eq:HI} .
\end{proof}

\section{Factorisation of the Dirichlet heat kernel}\label{sec:dhk}
This section is devoted to the proof of Theorem \ref{thm:dhk}. In what follows it will be crucial to exploit the geometric properties of $\kappa$-fat sets. Recall that $D$ is $\kappa$-fat, if there is $R>0$ such that for every $r \in (0,R)$ and $x \in \overline{D}$ there is a point $A_r(x)$ such that $B(A_r(x),\kappa r) \subseteq B(x,r) \cap D$. We then apply the following standard notation, c.f. \cite[Definition 2]{KBTGMR10} and \cite[Figure 1]{CKS14}. Namely, for $x \in D$ and $r>0$, we let  
\[
U^{x,r}=B(x,|x-A_r(x)|+\kappa r/3) \cap D, \qquad B^{x,r}_1 = B(A_r(x),\kappa r/3),
\]
so that $B_1^{x,r} \subset U^{x,r}$. We also set $A'_r(x)$ and $B^{x,r}_2 = B(A'_r(x),\kappa r/6)$ so that $B(A'_r(x),\kappa r/3) \subset B(A_r(x), \kappa r) \setminus U^{x,r}$ and consequently $\dist(U^{x,r},B_2^{x,r}) \geq \kappa r/6$. Finally we let $V^{x,r} = B(x,|x-A_r(x)|+\kappa r) \cap D$. The structures above are presented on Figure \ref{fig:1}.
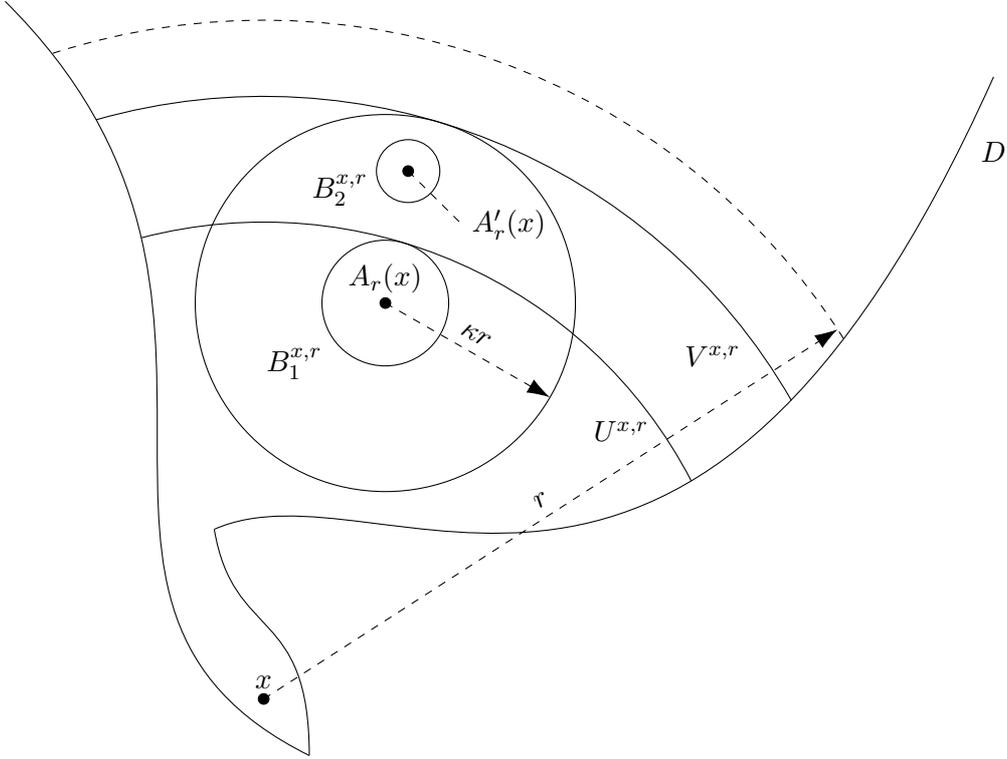
\begin{figure}[H]
	\centering
	\begin{tikzpicture}
		\draw[name path=set] (-1,2) .. controls (3,-2) and (-1,-6) .. (3,-8);
		\draw (3,-8) .. controls (3,-6) and (2,-6.5) .. (1.75,-5);
		\draw (1.75,-5) .. controls (4,-4) and (8,-8) .. (12,1);
		\filldraw (2.4,-7.25) circle (2pt) node[above] {$x$} coordinate (X);
		\draw[dashed, name path=arc] ([shift=(32.1:9)]X) arc (32.1:108:9);
		\draw[postaction={decorate, decoration={text along path, text={$r${}}, text align={center}, raise = 4pt}}, dashed, -{Latex[length=3mm,width=2mm]}] (X) -- +(33:9cm);
		\filldraw (4,-2) circle (2pt) node[above] {$A_r(x)$} coordinate (A);
		\draw (A) circle (2.5);
		\draw[postaction={decorate, decoration={text along path, text={${\kappa}r${}}, text align={center}, raise = 4pt}}, dashed, -{Latex[length=3mm,width=2mm]}] (A) -- +(330:2.5cm);
		\draw (A) circle (2.5/3); 
		\draw ([shift=(27.2:5.49+2.5/3)]X) arc (27.2:104.7:5.49+2.5/3);
		\draw ([shift=(29.7:5.49+2.5)]X) arc (29.7:106:5.49+2.5);
		\filldraw (4.3,-0.25) circle (2pt) node[above] {} coordinate (A1);
		\draw (A1) circle (2.5/6);
		\draw[dashed] (A1) -- +(315:1cm) node[right] {$A_r'(x)$};
		\node at (7.1,-3.7) {${U^{x,r}}$};
		\node at (8.3,-2.7) {${V^{x,r}}$};
		\node at (2.8,-2.8) {${B_1^{x,r}}$};
		\node at (3.4,-0.5) {${B_2^{x,r}}$};
		\node at (12,0) {$D$};
	\end{tikzpicture}
	\caption{Illustration of geometric structure of $\kappa$-fat sets}
	\label{fig:1}
\end{figure}
We start with a following observation on the survival probability.
\begin{lemma}\label{lem:surv_prop_est}
	There is a constant $c=c(\alpha,d,\lambda,\kappa)$ such that if $D$ is $(\kappa,1)$-fat in $x$, then
	\[
	\P_x(\tau_D>1/3) \leq c \P_x(\tau_D > 3).
	\]
\end{lemma}
\begin{proof}
	Let $x \in D$ and denote $A=A_1(x)$, $U = U^{x,1}$ and $B_2 = B_2^{x,1}$. For $|x-A|<\kappa/2$ we have $B(x,\kappa/2) \subseteq D$ and
	\[
	1 \geq \P_x(\tau_D>1/3) \geq \P_x(\tau_D>3) \geq \P_x\Big(\tau_{B_{(x,\kappa/2)}}>3\Big) = \P_0(\tau_{B_{\kappa/2}}>3)=c>0,
	\]
	with the constant $c$ depending on $\alpha$, $d$, $\lambda$ and $\kappa$. For $|x-A| \geq \kappa/2$ we write
	\begin{equation}\label{eq:2}
	\P_x(\tau_D>1/3) \leq \P_x(\tau_U>1/3) + \P_x(X_{\tau_U} \in D).
	\end{equation}
	Observe that for $B = B(x,|x-A|+\kappa/3)$ we have $U \subseteq B$ and consequently,
	\[
	\P_x(X_{\tau_U} \in \partial U \cap D) \leq \P_x(X_{\tau_B} \in \partial B) = 0,
	\]
	see e.g. \cite[Section 4]{Sztonyk00} or \cite[Proposition 5.2]{Vondracek02}. Therefore,
	\begin{equation*}
	\P_x(X_{\tau_U} \in D) = \P_x(X_{\tau_U} \in D \setminus \overline{U}).
	\end{equation*}
	Therefore, by \cite[Remark 3.8]{BKK15} with $x_0=x$, $y=A$, $R=|x-A|+\kappa/3$, $r=R-\kappa/6$, $E_1 = D \setminus \overline{U}$ and $E_2 = B_2$ in the notation of \cite[Remark 3.8]{BKK15}, we get that
	\[
	\frac{\P_x(X_{\tau_U} \in D)}{\P_A(X_{\tau_U} \in D)} \leq c\frac{\P_x(X_{\tau_U} \in B_2)}{\P_A(X_{\tau_U} \in B_2)}
	\]
	with $c=c(d,\alpha,\lambda,\kappa)$. Next, we infer from \eqref{eq:h} and Proposition \ref{prop:Pr_est} that $P_{B_1}(A,\,\cdot\,)$ is strictly positive on $\overline{B}_1^c$; therefore,
	\[
	\P_A(X_{\tau_U} \in B_2) \geq \P_A(X_{\tau_{B_1}} \in B_2) = c>0.
	\]
	Thus,
	\begin{equation}\label{eq:20}
	\P_x(X_{\tau_U} \in D) \leq c \P_x(X_{\tau_U} \in B_2).
	\end{equation}
	Moreover, note that $|y-u| \approx 1$ for $y \in B_2$ and $u \in U$. By the Ikeda-Watanabe formula \eqref{eq:IW} and \eqref{eq:nu_approx},
	\[
	\P_x (X_{\tau_U} \in B_2) = \int_U G_U(x,u)\int_{B_2} \nu(y-u)\ud y \ud u \approx \int_U G_U(x,u)\ud u = \E_x \tau_U. 
	\]
	The Markov inequality yields $\P_x(\tau_U>1/3) \leq 3\E_x \tau_U$. Combining with \eqref{eq:2} and \eqref{eq:20}, we get
	\[
	\P_x(\tau_D>1/3) \leq c\E_x \tau_U,
	\]
	and by the strong Markov property,
	\begin{align}\label{eq:22}
		\E_x\tau_U \leq c \P_x (X_{\tau_U} \in B_2) \leq c \E_x \Big[ X_{\tau_U} \in B_2; \P_{X_{\tau_U}}\big(\tau_{B(X_{\tau_U},\kappa/6)}>3\big) \Big] \leq c\P_x(\tau_D>3),
	\end{align}
	which ends the proof.
\end{proof}
\begin{remark}\label{rem:2} We note two observations from the proof of Lemma \ref{lem:dhk_partial_est}.
	\begin{enumerate}
		\item
		Let $V=V^{x,1}$. Then in \eqref{eq:22} one in fact has
		\[
		\E_x \tau_U \lesssim  \E_x \Big[ X_{\tau_U} \in B_2; \P_{X_{\tau_U}}\big(\tau_{B(X_{\tau_U},\kappa/6)}>3\big) \Big] \leq \P_x(\tau_V > 3) \leq \P_x (\tau_D>3).
		\]
		Clearly, if $D$ is $(\kappa,1)$-fat in $x$, it is $(\kappa/3,1)$-fat in $x$, too. Since $D \cap B(x,|x-A|+\kappa/3)=U$, by the proof above we obtain
		\begin{equation}\label{eq:23}
		\P_x(\tau_D>1/3) \approx \P_x(\tau_D>3) \approx \P_x(\tau_D>1) \approx \P_x(\tau_U>1) \approx \P_x(X_{\tau_U} \in D) \approx \E_x \tau_U.
		\end{equation}
		\item In fact, in \eqref{eq:23} we may replace $3$ by any constant $a\geq 1$ at the cost of the comparability constant depending also on $a$. We also note that the analogous results hold for the dual process $\hbfX^D$, too.
	\end{enumerate}

\end{remark}
\begin{lemma}\label{lem:dhk_partial_est}
	Let $D_1, D_3 \subseteq D$ be open and such that $\dist(D_1,D_3)>0$. Denote $D_2 = D \setminus (D_1 \cup D_3)$. If $x \in D_1$ and $y \in D_3$, then
	\begin{align*}
		p_1^D(x,y) \leq \P_x(X_{\tau_{D_1}} \in D_2) \sup_{s <1, z \in D_2}p_s(z,y) + \E_x \tau_{D_1} \sup_{u \in D_1, z \in D_3} \nu(z-u) \\ \intertext{and} \\ p_1^D(x,y) \geq \P_x(\tau_{D_1}>1) \hP_y(\tau_{D_3}>1) \inf_{u \in D_1, z \in D_3} \nu(z-u).
	\end{align*}
\end{lemma}
\begin{proof}
	By the strong Markov property,
	\begin{align*}
	p_1^D(x,y) &= \E_x \Big[ \tau_{D_1}<1;p_{1-\tau_{D_1}}^D\big(X_{\tau_{D_1}},y\big) \Big] \\ &= \E_x \Big[ \tau_{D_1}<1, X_{\tau_{D_1}} \in D_2;p_{1-\tau_{D_1}}^D\big(X_{\tau_{D_1}},y\big) \Big] + \E_x \Big[ \tau_{D_1}<1,X_{\tau_{D_1}}\in D_3;p_{1-\tau_{D_1}}^D\big(X_{\tau_{D_1}},y\big) \Big] \\ &=:\mathrm{I}_1 + \mathrm{I}_2.
	\end{align*}
	We clearly have
	\[
	\mathrm{I}_1 \leq \P_x\big(X_{\tau_{D_1}} \in D_2\big) \sup_{s <1, z \in D_2} p_s(z,y).
	\]
	Let $D_1$ be such that $\P_x(X_{\tau_{D_1}} \in \partial D_1 \cap D) = 0$, for instance when $D_1=D \cap E$ for some Lipschitz domain $E$. Then the density of $(\tau_{D_1},X_{\tau_{D_1}})$ at $(s,z)$ with $z \in D$ is given by (see  \eqref{eq:IW_density})
	\[
	f_x(s,z) = \int_{D_1} p_s^{D_1}(x,u)\nu(z-u)\ud u.
	\]
	Then, using \eqref{eq:dual_dhk}, for $z \in D_3$,
	\begin{equation}\label{eq:4}
	f_x(s,z) = \int_{D_1} p_s^{D_1}(x,u)\nu(z-u)\ud u \leq \P_x(\tau_{D_1}>s)\sup_{u \in D_1, z \in D_3} \nu(z-u),
	\end{equation}
	hence, by \eqref{eq:3} and \eqref{eq:dual_dhk},
	\begin{align*}
		\mathrm{I}_2 &= \int_0^1 \int_{D_3} p_{1-s}^D(z,y)f_x(s,z)\ud z\ud s \\ &\leq \sup_{u \in D_1, z \in D_3} \nu(z-u) \int_0^1 \int_{D_3} \widehat{p}_{1-s}^D(y,z) \P_x(\tau_{D_1}>s)\ud z\ud s \\ &\leq \int_0^1\P_x(\tau_{D_1}>s)\ud s \cdot \sup_{u \in D_1, z \in D_3} \nu(z-u) \\ &\leq \E_x \tau_{D_1} \cdot \sup_{u \in D_1, z \in D_3} \nu(z-u),
	\end{align*}
 	and the upper estimate follows in this case. For general $D_1$, we pick an approximating sequence $(E_n)_{n \in \N}$ such that $E_n \subseteq E_{n+1}$, $\cup_{n \in \N}E_n = D_1$ and $\P_x(X_{\tau_{E_n}} \in \partial E_n \cap D)=0$, and use the continuity of the heat kernel $p$. Note that the continuity of the L\'{e}vy density $\nu$ is not necessary.
 	
 	For the proof of the second part, we use the reverse estimate to \eqref{eq:4} together with \eqref{eq:dual_dhk} to obtain
 	\begin{align*}
 		\mathrm{I}_2 &\geq \inf_{u \in D_1, z \in D_3} \nu(z-u) \int_0^1 \int_{D_3} p_{1-s}^D(z,y) \P_x(\tau_{D_1}>s) \ud z\ud s \\ &\geq \P_x(\tau_{D_1}>1) \inf_{u \in D_1, z \in D_3} \nu(z-u) \int_0^1 \int_{D_3} \widehat{p}_{1-s}^{D_3}(y,z) \ud z\ud s \\ &\geq  \P_x(\tau_{D_1}>1) \hP_y(\tau_{D_3}>1) \inf_{u \in D_1, z \in D_3} \nu(z-u).
 	\end{align*}
\end{proof}
\begin{lemma}\label{lem:dhk_upper}
	If $D$ is ($\kappa$,1)-fat at $x$ and $y$, then there is a constant $c=c(\alpha,d, \lambda,\kappa)$ such that
	\[
	p_2^D(x,y) \leq c \P_x(\tau_D > 2) p_2(x,y) \hP_y(\tau_D > 2).
	\]
\end{lemma}
\begin{proof}
	We first claim that if $D$ is $(\kappa,1)$-fat at $u$, then there is $c = c(\alpha,d,\lambda,\kappa)$ such that 
	\begin{align}
			p_1^D(u,v)  \leq cp_1(u,v) \P_u(\tau_D>1), \quad v \in D, \label{eq:5} \\ \intertext{and} \widehat{p}_1^D(u,v) \leq cp_1(v,u)\hP_u (\tau_D>1), \quad v \in D. \label{eq:6}
	\end{align}
	Indeed, for $|u-v| \leq 8$ we have $p_1(u,v) \approx p_1(v,u) \approx 1$ (see \eqref{eq:10}); hence, by the semigroup property, \eqref{eq:hk_approx} and Lemma \ref{lem:surv_prop_est},
	\begin{align*}
	p_1^D(u,v) &= \int_D p_{1/2}^D(u,z)p_{1/2}^D(z,v)\ud z \\ &\leq \sup_{z \in \Rd} p_{1/2}(z,v) \P_u(\tau_D>1/2) \\ &\leq c \P_u(\tau_D>1) \\ &\leq cp_1(u,v) \P_u(\tau_D>1),
	\end{align*}
	and \eqref{eq:5} follows in this case. Therefore, we assume that $|u-v|>8$ and use Lemma \ref{lem:dhk_partial_est} with $A=A_1(u)$, $D_1 = U^{u,1} = D \cap B(A,|u-A|+\kappa/3)$ and $D_3 = \Big\{z \in D \colon |z-u|>\frac12|u-v|\Big\}$. Observe that by \eqref{eq:nu_approx}, \eqref{eq:hk_approx} and \eqref{eq:11},
	\[
	\sup_{s <1 , z \in D_2} p_s(z,v) \leq cp_1(u,v)
	\]
	and
	\[
	\sup_{w \in D_1, z \in D_3} \nu(z-w) \leq cp_1(u,v).
	\]
	Hence, by Lemma \ref{lem:dhk_partial_est},
	\begin{align*}
	p_1^D(u,v) &\leq cp_1(u,v) \big( \P_u(X_{\tau_{D_1}} \in D_2) + \E_u \tau_{D_1} \big) \\ &\leq cp_1(u,v) \big( \P_u(X_{\tau_{D_1}} \in D) + \E_u \tau_{D_1} \big) \\ &\leq cp_1(u,v) \P_u(\tau_D>1),
	\end{align*}
	where the last inequality is a consequence of Remark \ref{rem:2}. For the proof of \eqref{eq:6} it remains to consider the dual process $\hbfX$, use \eqref{eq:10} and proceed in exactly the same way. Thus, by the semigroup property, \eqref{eq:dual_dhk}, \eqref{eq:5} and \eqref{eq:6}, and Lemma \ref{lem:surv_prop_est} with \eqref{eq:10},
	\begin{align*}
		p_2^D(x,y) &= \int_D p^D_1(x,z)p_1^D(z,y)\ud z \\ &= \int_D p^D_1(x,z)\widehat{p}_1^D(y,z)\ud z \\ &\leq c\P_x(\tau_D>1) \hP_y(\tau_D>1) \int_Dp_1(x,z)p_1(z,y)\ud z \\ &\leq c \P_x(\tau_D>2) \hP_y(\tau_D>2)p_2(x,y).
	\end{align*}
\end{proof}
\begin{remark}\label{rem:1}
	In fact, under the assumptions of Lemma \ref{lem:dhk_upper} we get that
	\begin{equation}\label{eq:7}
		p_1^D(x,y) \leq c\P_x(\tau_D>1) \hP_y(\tau_D > 1)p_1(x,y).
	\end{equation}
	For the proof let us consider the modified L\'{e}vy measure $\widetilde{\nu} = \frac12 \nu$, the corresponding kernels $\widetilde{p}$ and $\widetilde{p}^D$, and the probability distribution $\widetilde{\P}$. Then by elementary calculations one gets that $\widetilde{p}_t^D(x,y) = p_{t/2}^D(x,y)$ for all $x,y \in D$ and all $t>0$. It follows from Lemma \ref{lem:dhk_upper} that
	\begin{align*}
		p_1^D(x,y) &= \widetilde{p}_{2}^D(x,y) \\ &\leq \widetilde{c}\widetilde{\P}_x(\tau_D>2)\widetilde{\hP}_y(\tau_D>2)\widetilde{p}_2(x,y) \\ &= \widetilde{c}\P_x(\tau_D>1)\hP_y(\tau_D>1)p_1(x,y).
	\end{align*}
\end{remark}

Now we deal with the lower estimate of the Dirichlet heat kernel.
\begin{lemma}\label{lem:lower_est_balls}
	Let $r>0$. There is $c=c(\alpha,d,\lambda,r)$ such that
	\[
	p_1^{B(u,r) \cup B(v,r)}(u,v) \geq cp_1(u,v), \quad u,v \in \Rd.
	\]
\end{lemma}
\begin{proof}
	If $|u-v|>r/2$, then we apply Lemma \ref{lem:dhk_partial_est} with $D = B(u,r) \cup B(v,r)$, $D_1 = B(u,r/8)$ and $D_3 =B(v,r/8)$, so that
	\begin{align*}
	p_1^{B(u,r)\cup B(v,r)}(u,v) &\geq \P_u(\tau_{D_1}>1)\hP_y(\tau_{D_3}>1) \inf_{w \in D_1, z \in D_3} \nu(z-w) \\ &\geq \P_0(\tau_{B_{r/8}}>1)\hP_0(\tau_{B_{r/8}}>1)p_1(u,v) \\ &= c p_1(u,v),
	\end{align*}
	where the second inequality follows from \eqref{eq:10}.	The case $|u-v| \leq r/2$ is even simpler; we then have by continuity and strict positivity of $p_1^{B_r}$ (see \cite[Theorem 3.2(4)]{Vondracek02}) together with \eqref{eq:11} that
	\[
	p_1^{B(u,r)\cup B(v,r)}(u,v) \geq \inf_{|z|<r/2}p_1^{B_r}(0,z) \geq c \geq cp_1(u,v).
	\]
\end{proof}
\begin{lemma}\label{lem:dhk_lower}
	If $D$ is $(\kappa,1)$-fat at $x$ and $y$, then there is $c = c(\alpha,d,\lambda,\kappa)$ such that
	\[
	p_3^D(x,y) \geq c \P_x(\tau_D>3)p_3(x,y)\hP_y(\tau_D>3).
	\]
\end{lemma}
\begin{proof}
	By the semigroup property, Lemma \ref{lem:lower_est_balls} with $r=\kappa/6$ and \eqref{eq:11}, we have
	\begin{align*}
		p_3^D(x,y) &\geq \int_{B_2^{x,1}}\int_{B_2^{y,1}}p_1^D(x,u)p_1^D(u,v)p_1^D(v,y)\ud u\ud v \\ &\geq cp_1(x,y) \int_{B_2^{x,1}}p_1^D(x,u)\ud u \int_{B_2^{y,1}}\widehat{p}_1^D(y,v)\ud v.
	\end{align*}
	For $u \in B_2^{x,1} = B(A'_1(x),\kappa/6)$, using Lemma \ref{lem:dhk_partial_est} with $D_1 = U^{x,1}$ and $D_3 = B(A'_1(x),\kappa/4)$, \eqref{eq:nu_approx} and Remark \ref{rem:2}, we get
	\begin{align*}
		p_1^D(x,u) &\geq \P_x(\tau_{D_1}>1) \hP_u(\tau_{D_3}>1) \inf_{w \in D_1, z \in D_3} \nu(z-w) \\ &\geq c\P_x(\tau_{D_1}>1) \hP_0(\tau_{B_{\kappa/12}}>1) \\ &\geq c\P_x(\tau_D>1).
	\end{align*}
	In the same way we obtain, for $v \in B_2^{y,1}$,
	\[
	\widehat{p}_1^D(y,v) \geq c \hP_y(\tau_D>1).
	\]
	Therefore, by Lemma \ref{lem:surv_prop_est} and \eqref{eq:12},
	\begin{align*}
	p_3^D(x,y) &\geq c \P_x(\tau_D>1) p_1(x,y) \hP_y(\tau_D>1) \\ &\geq c \P_x(\tau_D>3) p_3(x,y) \hP_y(\tau_D>3).
	\end{align*}
\end{proof}
\begin{remark}
	Proceeding exactly as in Remark \ref{rem:1} we conclude that under the assumptions of Lemma \ref{lem:dhk_lower} it holds that  
	\begin{equation}\label{eq:8}
		p_1^D(x,y) \geq c \P_x(\tau_D>1)\hP_y(\tau_D>1)p_1(x,y).
	\end{equation}	
\end{remark}

\begin{proof}[Proof of Theorem \ref{thm:dhk}]
	Suppose first that $R \in (0,\infty)$ and $c = 1$. Then it follows that $t^{-1/\alpha}D$ is $(\kappa,1)$-fat and the result follows from \eqref{eq:7}, \eqref{eq:8} and scaling properties \eqref{eq:hk_scaling}, \eqref{eq:dhk_scaling} and \eqref{eq:surv_scaling} with $C = C(\alpha,d,\lambda,\kappa)$. For $c>1$ we use the fact that if $D$ is $(\kappa,R)$-fat, then it is also $(\kappa/a,aR)$-fat for any $a \in [1,\infty)$. Thus, setting $a=c^{1/\alpha}$ we arrive at the previous setting at the cost of worsening the constant. In this case $C = C(\alpha,d,\lambda,\kappa,c)$.
	
	For the case $R=\infty$ we note that by the convention adopted in Section \ref{sec:prelims}, $D$ is $(\kappa,r)$-fat for all $r>0$ with $\kappa$ independent of $r$. Thus, for any $t>0$ one can pick $r>0$ such that $r^\alpha \geq t$ and apply the reasoning above to conclude the claim. This time $c=1$ and $C=C(\alpha,d,\lambda,\kappa)$.
\end{proof}
\section{Stable processes in generalised $\kappa$-fat cones}\label{sec:Y}
Now we specify our analysis to open sets $\Gamma\subseteq \Rd$ which are invariant under rescaling, i.e. for every $r>0$, $ry \in \Gamma$ whenever $y \in \Gamma$. In other words, from this moment on, $\Gamma$ is a generalised cone in $\Rd$. Thus, if $0 \in \Gamma$, then necessarily $\Gamma=\Rd$. Otherwise, $\Gamma$ can be characterised by its intersection with the unit sphere $\Sd$. For example, for $d=1$ we have four possibilities: $\Gamma=\Rd$, $\Gamma=\R \setminus \{0\}$, $\Gamma = (0,\infty)$ and $\Gamma = (-\infty,0)$. We note that by the scaling invariance, if $\Gamma$ is $(\kappa,r)$-fat for some $r>0$, then it is in fact $(\kappa,r)$-fat for every $r>0$ with the constant $\kappa$ independent of $r$. 

In what follows, it will be convenient to fix a reference point $\one=(0,\ldots,0,1) \in \Rd$. Note that if $\one \notin \Gamma$, then one can consider an appropriate rotation matrix $U \in \calM_d$ such that $\one \in \{Ux: x \in \Gamma\}$. It is then easy to see that $U\bfX:=(UX_t\colon t \geq 0)$ is also a strictly stable L\'{e}vy process with the spherical measure $\widetilde{\zeta}$ given by $\widetilde{\zeta}(B) = \zeta(\{x \colon Ux \in B\})$ and the drift component $\widetilde{\gamma}:=U\gamma + \int Ux(\ind_{B_1}(Ux) - \ind_{B_1}(x))\,\nu(\od x)$, c.f. \cite[Proposition 11.10]{Sato99}. Clearly, if $\bfX$ satisfies \eqref{eq:main_as}, then so does $U\bfX$ (with the same constant $\theta$). Thus, at the possible cost of \emph{rotating} the cone $\Gamma$ and the underlying process $\bfX$, we may and do assume that $\one \in \Gamma$.
\subsection{Martin kernel for the cone and its properties}

First, we collect and apply results from general theory \cite{KSV17, KSV18} to establish the existence and basic properties of the Martin kernel for arbitrary $\kappa$-fat cones $\Gamma$. Given $R>0$, we define $\Gamma_R:= \Gamma \cap B_R$. The following is a version of \cite[Theorem 3.2]{BB04} for anisotropic $\alpha$-stable processes.
\begin{lemma}\label{lem:M}
	Assume $\Gamma$ is a $\kappa$-fat cone in $\Rd$ and suppose $\bfX$ is a strictly stable L\'{e}vy process satisfying {\bf A}. There is a unique non-negative function on $\Rd$ such that $M_{\Gamma}(\one)=1$, $M_{\Gamma}(x)=0$ for $x \in \Gamma^c$ and $M_{\Gamma}$ is regular harmonic with respect to $\bfX$ in every open bounded subset $B \subseteq \Gamma$, i.e.
	\begin{equation*}
		M_{\Gamma}(x) = \E_x M_{\Gamma} \big( X_{\tau_B} \big), \quad x \in \Gamma.
	\end{equation*}
	Moreover, $M_{\Gamma}$ is locally bounded and homogeneous of order $\beta = \beta(\alpha,\lambda,\Gamma) \in [0,\alpha)$, i.e.
	\begin{equation}\label{eq:Mk_hom}
		M_{\Gamma}(x) = |x|^{\beta}M_{\Gamma}(x/|x|), \quad x \in \Gamma.
	\end{equation}
	If $\Gamma$ is Greenian, then one also has
	\begin{equation*}
	M_{\Gamma}(x) = \lim_{\Gamma \ni y,|y| \to \infty} \frac{G_{\Gamma}(x,y)}{G_{\Gamma}(\one,y)}.
	\end{equation*}

\end{lemma}
The function $M_{\Gamma}$ is called the Martin kernel with pole at infinity for the cone $\Gamma$ and the process $\bfX$.

\begin{proof}[Proof of Lemma \ref{lem:M}]
	We first establish the existence of a function satisfying the desired properties. If $d=1$ and $\alpha=1$, then singletons are polar sets (see, for example, \cite[Chapter 8]{Sato99}). Thus, for $\Gamma=\R$ or $\Gamma=\R \setminus \{0\}$ we see that $M_{\Gamma}=\ind_{\Gamma}$ satisfies the required properties with $\beta=0$. The same candidate works for $d=1$, $\alpha \in (1,2)$ and $\Gamma=\R$. Thus, we may assume that $\Gamma$ is Greenian.
	
	Since $G_{\Gamma}(x,y)=0$ when $x \in \Gamma^c$ and $y \in \Rd$, we may and do assume that $x \in \Gamma$. By \cite[Example 5.1]{KSV17}, $\bfX$ satisfies all the assumptions stated in \cite{KSV17}; thus, by \cite[Proposition 4.1, Remark 4.2(b) and the discussion below it]{KSV18}, the infinity is accessible from $\Gamma$ with respect to $\bfX$, i.e. $\E_x \tau_{\Gamma} = \infty$ for every $x \in \Gamma$. Then it follows from \cite[Theorem 1.3(b)]{KSV18} that the limit
	\[
	M_{\Gamma}(x) = \lim_{\Gamma \ni y,|y| \to \infty} \frac{G_{\Gamma}(x,y)}{G_{\Gamma}(\one,y)}
	\]
	exists and is finite for every $x \in \Gamma$. Moreover, by \cite[proof of Theorem 1.3(b)]{KSV18}, for every bounded open set $B \subseteq \Gamma$,
	\[
	M_{\Gamma}(x) = \E_x M_{\Gamma}(X_{\tau_B}), \quad x \in B.
	\]
	That is, $M_{\Gamma}$ is regular harmonic in $B$. The existence part is thus established.
	
	In the remaining part we follow the proof of \cite[Theorem 3.2]{BB04} to derive its basic properties. First, we will show that $M_{\Gamma}$ is locally bounded and homogeneous of order $\beta = \beta(\alpha,\lambda,\Gamma) \in [0,\alpha)$, i.e.
	\begin{equation}\label{eq:M_hom}
	M_{\Gamma}(x) = |x|^{\beta}M_{\Gamma}(x/|x|), \quad x \in \Gamma.
	\end{equation}
	Indeed, in the non-Greenian case one has $M_{\Gamma} = \ind_{\Gamma}$ as above. Then $\beta=0$ and $M_{\Gamma}$ is locally bounded. Thus, we assume that $\Gamma$ is Greenian, fix $R>0$ and let $x \in \Gamma_R$. Since $G_{\Gamma}(x,\,\cdot\,)$ is regular co-harmonic in $\Gamma \setminus \Gamma_{8R}$, by the boundary Harnack principle at infinity \cite[Corollary 2.2, Remark 2.3 and Example 5.1]{KSV17}, there is $c>0$ such that for some fixed $y_0 \in \Gamma \setminus \Gamma_{8R}$,
	\[
	\frac{G_{\Gamma}(x,y)}{G_{\Gamma}(\one,y)} \leq c \frac{G_{\Gamma}(x,y_0)}{G_{\Gamma}(\one,y_0)}, \quad y \in \Gamma \setminus \Gamma_{8R}.
	\]
	Thus, $M_{\Gamma}$ is locally bounded.

	Next, recall that $k\Gamma=\Gamma$ for every $k>0$. Therefore, by the scaling property of $G_{\Gamma}$ \eqref{eq:GD_scaling}, for every $x,y \in \Gamma$,
	\[
	\frac{G_{\Gamma}(kx,y)}{G_{\Gamma}(\one,y)} \cdot \frac{G_{\Gamma}(\one,y)}{G_{\Gamma}(k\one,y)} = 
	\frac{G_{\Gamma}(kx,y)}{G_{\Gamma}(k\one,y)} = \frac{k^{-d+\alpha}{G_{\Gamma}(x,k^{-1}y)}}{k^{{-d+\alpha}} G_{\Gamma}(\one,k^{-1}y)} = \frac{{G_{\Gamma}(x,k^{-1}y)}}{G_{\Gamma}(\one,k^{-1}y)}.
	\]
	By taking the limit as $|y| \to \infty$ we see that
	\[
	M_{\Gamma}(kx) = M_{\Gamma}(x) M_{\Gamma}(k\one).
	\]
	It follows that
	\[
	M_{\Gamma}(kl\one) = M_{\Gamma}(k\one)M_{\Gamma}(l\one),
	\]
	for every positive $k,l$. Note that in view of \cite[Theorem 6.7]{Vondracek02}, $M_{\Gamma}$ is continuous on $\Gamma$. Thus, there is $\beta \in \R$ such that $M_{\Gamma}(k\one) = k^{\beta}M_{\Gamma}(\one)=k^{\beta}$ and hence,
	\[
	M_{\Gamma}(kx) = k^{\beta}M_{\Gamma}(x), \quad x \in \Gamma.
	\]
	The local boundedness of $M_{\Gamma}$ implies that $\beta \geq 0$ and Lemma \ref {lem:harmonic_growth_bound} entails that $\beta \leq \gamma < \alpha$. Thus, $\beta \in [0,\alpha)$.
	
	To conclude the proof, it remains to note that the uniqueness may be verified using the boundary Harnack principle \eqref{eq:BHP} exactly as in the proof of \cite[Theorem 3.2]{BB04}.
\end{proof}

\begin{remark}\label{rem:4}
	By inspecting the proof above, one can quickly verify that exactly the same reasoning may be applied to the dual process $\hbfX$. Therefore, the Martin kernel $\hM_{\Gamma}$ for the cone exists, too, and the statements of Lemma \ref{lem:M} and Corollary \ref{cor:surv} remain in force, when one replaces $G_{\Gamma}$ by $\widehat{G}_{\Gamma}$, $M_{\Gamma}$ by $\hM_{\Gamma}$ and $\beta$ by $\hbeta$. Note that Lemma \ref{lem:M} does not give information on the relation between $\beta$ and $\hbeta$, c.f. Remark \ref{rem:5}\eqref{rem:5_2} and Example \ref{ex:halfspace}.
\end{remark}
The following version of \cite[Theorem 2]{KBTGMR10} is now an immediate corollary of Lemma \ref{lem:M} and Remark \ref{rem:4}.
\begin{corollary}\label{cor:surv}
	Suppose $\Gamma$ is a $\kappa$-fat cone. Then for all $t>0$ and $x,y \in \Gamma$, we have
	\[
	\P_x(\tau_{\Gamma}>t) \approx \frac{M_{\Gamma}(x)}{M_{\Gamma}(A_{t^{1/\alpha}}(x))} \qquad \text{and} \qquad \hP_x(\tau_{\Gamma}>t) \approx \frac{\hM_{\Gamma}(y)}{\hM_{\Gamma}(A_{t^{1/\alpha}}(y))}.
	\]
	Furthermore,
	\[
	p_t^{\Gamma}(x,y) \approx \frac{M_{\Gamma}(x)}{M_{\Gamma}(A_{t^{1/\alpha}}(x))} p_t(x,y) \frac{\hM_{\Gamma}(y)}{\hM_{\Gamma}(A_{t^{1/\alpha}}(y))}, \quad t>0,\,x,y \in \Gamma.
	\]
	The comparability constants depend at most on $d,\alpha, \kappa$ and $\lambda$.
\end{corollary}
\begin{proof}
	Recall that $\Gamma$ is $(\kappa,r)$-fat for every $r>0$ with the constant $\kappa$ independent of $r$. Thus, with Lemma \ref{lem:M}, Remark \ref{rem:2}, \eqref{eq:BHP} and Theorem \ref{thm:dhk} at hand, the proof is a mimic of the proof of \cite[Theorem 2]{KBTGMR10} and is therefore omitted. The dual counterpart follows by Remark \ref{rem:4}.
\end{proof}
\begin{remark}\label{rem:5}
	\hspace{2em}
	\begin{enumerate}
		\item\label{rem:5_1} Proceeding as at the end of the proof of \cite[Theorem 3.2]{BB04} with the help of \cite[Proposition 6.1(3)]{Vondracek02}, one can see that $\beta=0$ if and only if $\Gamma^c$ is polar for $\bfX$. Moreover, the following \emph{monotonicity} of the homogeneity index holds: if $\gamma,\Gamma$ are $\kappa$-fat cones in $\Rd$ and $\gamma \subseteq \Gamma$, then $\beta(\alpha,\lambda,\Gamma) \geq \beta(\alpha,\lambda,\gamma)$ and the equality holds if and only if $\Gamma \setminus \gamma$ is a polar set for $\bfX$. This is proved exactly as in \cite[Lemma 3.3]{BB04}. 
		\item\label{rem:5_2} By \cite[Theorem 42.30]{Sato99}, $\Gamma^c$ is polar for $\bfX$ if and only if it is polar for the dual process $\hbfX$. Therefore, using part \eqref{rem:5_1} we conclude that $\beta=0$ if and only if $\hbeta=0$. 
	\end{enumerate}
	
\end{remark}
\begin{remark}
	By \cite[Theorem 6.7]{Vondracek02}, $M_{\Gamma}$ is continuous in $\Gamma$.
\end{remark}

Let us discuss some examples, where the Martin kernel can be explicitly identified. Note that even in the isotropic setting, there are but a few cases when the exponent $\beta$ and the formula for $M_{\Gamma}$ are explicitly known, c.f. \cite{BB04} and \cite{Michalik06}. If we step outside the isotropic world, then the situation gets more complicated. The simplest case was already hinted in the proof of Lemma \ref{lem:M}.
\begin{example}
	Let $\Gamma$ be $\kappa$-fat and such that $\Gamma^c$ is a polar set. Then $M_{\Gamma}=\ind_{\Gamma}$ satisfies the required properties and by Lemma \ref{lem:M}, it is the Martin kernel for $\Gamma$. This class includes $\Gamma=\Rd$ for $\alpha \in (0,2)$ and $\Gamma = \Rd \setminus \{x_d=0\}$ for $\alpha \in (0,1]$, since then by \cite[Theorem 42.30]{Sato99}, $\{x_d=0\}$ is a polar set.
\end{example}
 Now let $\Gamma$ be the half-space $\H_d:= \{x \in \Rd \colon x_d>0\}$ in $\Rd$. In the rotation-invariant case, we have by \cite[Example 3.2]{BB04} that $M_{\Gamma}(x)=\delta_{\Gamma}(x)^{\alpha/2}$ and $\beta = \alpha/2$. In view of the boundary Harnack inequality, this exponent is in line with the typical rate of decay of harmonic functions in smooth domains, c.f. Chen and Song \cite{MR1654824} and Kulczycki \cite{Kulczycki99}. It turns out that dropping the symmetry assumption results in decay rate dependent on the asymmetry of the process, c.f. Juszczyszyn \cite[Theorem 1.1]{TJ21}. It also reflects in the homogeneity order of $M_{\Gamma}$, as we argue below. Recall that every finite measure $\zeta$ on the sphere $\Sd$ induces a strictly $\alpha$-stable L\'{e}vy process on $\Rd$. 
\begin{example}\label{ex:halfspace}
	 Suppose that $\zeta$ has a density $\lambda$ with respect to the surface measure $\sigma$, which is H\"{o}lder continuous on $\Sd$ of (some fixed) order $\epsilon>0$. Let $\Gamma = \H_d$ be a half-space in $\Rd$. By \cite[Lemma 2.16]{TJ21}, a function $h(x):=\delta_{\Gamma}(x)^{\gamma}$ is regular harmonic in every bounded open subset of $\H_d$, where $\gamma = \alpha \P(\langle X_1, \one \rangle > 0)$. By Lemma \ref{lem:M}, it is a Martin kernel for $\Gamma$ for the strictly $\alpha$-stable L\'{e}vy process $\bfX$ corresponding to the spherical density $\lambda$, with the homogeneity order $\beta = \alpha \P(\langle X_1, \one \rangle > 0)$.
	
	Let us take a closer look on the exponent $\beta$. It is clear that if $\bfX$ is symmetric (not necessarily isotropic), then $\beta=\alpha/2$ as in \cite[Example 3.2]{BB04}. In general, the process $Y_t:= \langle X_t,\one \rangle$ is a one-dimensional strictly $\alpha$-stable L\'{e}vy process with the (one-dimensional) L\'{e}vy density $\nu(z) = c_- |z|^{-1-\alpha}\ind_{z<0}+c_+|z|^{-1-\alpha}\ind_{z>0}$, where
	\[
	c_- = \int_{\Sd \cap \{x_d>0\}}\lambda(-w)w_d^{\alpha}\ud w \qquad \text{and} \qquad c_+ = \int_{\Sd \cap \{x_d>0\}}\lambda(w)w_d^{\alpha}\ud w,
	\]
	see e.g. \cite[Lemma 2.12]{TJ21}. Let $\alpha \neq 1$. By Zolotarev \cite{Zolotarev57}, one has
	\[
	\beta = \alpha \P(Y_1>0) =  \frac12 + \frac{1}{\pi\alpha} \arctan \bigg( \frac{c_+-c_-}{c_++c_-} \tan \frac{\pi \alpha}{2} \bigg).
	\]
	In particular, choosing $\lambda$ so that $c_+ \neq c_-$ yields $\beta \neq \alpha/2$.
\end{example}
The feature presented in Example \ref{ex:halfspace} has one expected, but nonetheless interesting consequence. We still assume $\Gamma = \H_d$ and let $\hM_{\Gamma}$ be the Martin kernel for the dual process $\hbfX$ for $\Gamma$ with the homogeneity order $\hbeta$. Continuing with the notation from Example \ref{ex:halfspace}, it is clear $\widehat{\lambda}(w) = \lambda(-w)$ for $w \in \Sd$. Thus, $\widehat{c}_- = c_+$, $\widehat{c}_+=c_-$ and consequently, $\hbeta \neq\beta$ unless $c_+ = c_-$. Put differently, even in the simple case of the half-space, the homogeneity (growth) orders $\beta$ and $\hbeta$ may be different. Consequently, in general one cannot expect the same exponents (that is: $\beta = \hbeta$) for dual processes $\bfX$ and $\hbfX$. 


We continue with Example \ref{ex:halfspace} and use it to derive explicit estimates on the Dirichlet heat kernel for $\H_d$.
\begin{example}\label{ex:halfspace2}
	Again let $\Gamma=\H_d$, $t>0$ and suppose that $x,y \in \Gamma$. Note that for every $r>0$ we have $\delta_{\Gamma} (A_r(x)) \approx r \vee \delta_{\Gamma}(x)$. Corollary \ref{cor:surv} together with Example \ref{ex:halfspace} then imply 
	\[
	\P_x(\tau_{\Gamma}>t) \approx \frac{M_{\Gamma}(x)}{M_{\Gamma}(A_{t^{1/\alpha}}(x))}  = \bigg(\frac{\delta_{\Gamma}(x)}{t^{1/\alpha} \vee \delta_{\Gamma}(x)}\bigg)^\beta \approx \bigg(1 \wedge \frac{\delta_{\Gamma}(x)}{t^{1/\alpha}}\bigg)^{\beta}, \quad t>0,\,x \in \Gamma.
	\]
	Similarly,
	\[
	\hP_y(\tau_{\Gamma}>t) \approx \bigg(1 \wedge \frac{\delta_{\Gamma}(y)}{t^{1/\alpha}}\bigg)^{\hbeta}, \quad t>0, \,y \in \Gamma.
	\]
	Therefore, again by Corollary \ref{cor:surv}, the global explicit two-sided estimate holds:
	\[
	p_t^{\Gamma}(x,y) \approx \bigg(1 \wedge \frac{\delta_{\Gamma}(x)}{t^{1/\alpha}}\bigg)^\beta p_t(x,y) \bigg(1 \wedge \frac{\delta_{\Gamma}(y)}{t^{1/\alpha}}\bigg)^{\hbeta}, \quad t>0, \,x,y \in \Gamma.
	\]
\end{example}
As a side remark, we observe that in this example one has $\beta + \hbeta = \alpha \P(Y_1 \neq 0)$. Thus, by \cite[Theorem 42.30]{Sato99}, $\beta + \hbeta=\alpha$ if and only if $\alpha \in (0, 1]$. As a side remark we note that in view of \cite[Lemma 2.14]{TJ21}, in this case there are constants $\beta_{{\rm min}},\beta_{{\rm max}}$ satisfying $\beta_{{\rm min}} > \max \{0,\alpha-1\}$, $\beta_{{\rm max}} < \min \{\alpha,1\}$, $\beta_{{\rm min}}+\beta_{{\rm max}}=\alpha$ and such that  $\beta,\hbeta \in [ \beta_{{\rm min}}, \beta_{{\rm max}}]$.

\subsection{Yaglom limit in cones}
In the remaining part of the paper we apply the method developed in \cite{un_BKLP23} to obtain the Yaglom limit for non-symmetric strictly $\alpha$-stable L\'{e}vy processes in arbitrary $\kappa$-fat cones. Its versatility lies in the fact that once the necessary tools are available, the existence of the limit in Theorem \ref{thm:Yaglom} follows by a similar argument. These basic blocks are: boundary Harnack inequality, sharp two-sided estimates on the Dirichlet heat kernel of the cone and nice scaling properties of a $P_t^{\Gamma}$-invariant function, which is later used for the version of Doob conditioning of the process killed when exiting the cone $\Gamma$. In our case, the first brick follows from Bogdan et al. \cite{BKK15}, the second is provided by Theorem \ref{thm:dhk}, and the third is a consequence of properties of the Martin kernel $M_{\Gamma}$ for the cone, which are gathered in Lemma \ref{lem:M}. With these tools at hand, we follow step by step the procedure described in \cite{un_BKLP23}. For the convenience of the reader, we provide all the details, unless the extension is immediate and requires no changes in the proof.

Let $\Gamma$ be a fixed $\kappa$-fat cone in $\Rd$ such that $\one \in \Gamma$. Recall that by Lemma \ref{lem:M}, $M_{\Gamma}$ is homogeneous of order $\beta \in [0,\alpha)$ and regular harmonic in every bounded subset of $\Gamma$. Thus, following \cite[Theorem 3.1]{un_BKLP23} one can prove that $M_{\Gamma}$ is invariant for the killed semigroup $P_t^{\Gamma}$, i.e. for every $t>0$ and $x \in \Gamma$,
\begin{equation}\label{eq:M_inv}
P_t^{\Gamma}M_{\Gamma}(x) = M_{\Gamma}(x).
\end{equation}
These two properties justify the following definition of the (version of) conditioned kernel
\begin{equation}\label{eq:rho_def}
\rho_t(x,y) = \frac{p_t^{\Gamma}(x,y)}{M_{\Gamma}(x)\hM_{\Gamma}(y)}, \quad x,y \in \Gamma,\,t>0.
\end{equation}
Using \eqref{eq:M_inv} it is easy to see that
\begin{equation}\label{eq:rho_density}
\int_{\Gamma} \rho_t(x,y)M_{\Gamma}(y)\hM_{\Gamma}(y)\ud y = 1, \quad x \in \Gamma,\,t>0,
\end{equation}
and
\begin{equation}\label{eq:rho_CH-K}
\int_{\Gamma} \rho_t(x,y)\rho_s(y,z) M_{\Gamma}(y)\hM_{\Gamma}(y)\ud y = \rho_t(x,z), \quad x,z \in \Gamma,\,t,s>0.
\end{equation}
Thus, $\rho_t$ is a transition probability density on $\Gamma$ with respect to the measure $M_{\Gamma}(y)\hM_{\Gamma}(y)\ud y$. We note that $\rho_t$ \emph{need not} be symmetric. By \eqref{eq:dhk_scaling} and Lemma \ref{lem:M}, the following scaling property holds for $\rho$: for all $x,y \in \Gamma$ and all $t>0$,
\begin{equation}\label{eq:rho_scaling}
\rho_t(x,y) = \frac{t^{-d/\alpha}p_1^{\Gamma}\big(t^{-1/\alpha}x,t^{-1/\alpha}y\big)}{t^{(\beta+\hbeta)/\alpha}M_{\Gamma}\big(t^{-1/\alpha}x\big) \hM_{\Gamma}\big(t^{-1/\alpha}y\big)} = t^{-(d+\beta+\hbeta)/\alpha}\rho_1 \big( t^{-1/\alpha}x,t^{-1/\alpha}y \big).
\end{equation}
Put differently,
\begin{equation}\label{eq:rho_scaling2}
\rho_{st} \big( t^{1/\alpha}x,t^{1/\alpha}y \big) = t^{-(d+\beta+\hbeta)/\alpha} \rho_s(x,y), \quad x,y \in \Gamma, \,s,t>0.
\end{equation}
By Theorem \ref{thm:dhk}, we also have
\begin{equation}\label{eq:rho_f}
\rho_t(x,y) \approx \frac{\P_x(\tau_{\Gamma}>t)}{M_{\Gamma}(x)}p_t(x,y) \frac{\hP_y(\tau_{\Gamma}>t)}{\hM_{\Gamma}(y)}, \quad x,y \in \Gamma, \,t>0.
\end{equation}
The structure of the factorisation \eqref{eq:rho_f} is one of reasons why we choose to define $\rho_t$ as in \eqref{eq:rho_def} instead of classical Doob h-transform using the invariant function $M_{\Gamma}$.

In what follows, it will be crucial to control the expression $\P_x(\tau_{\Gamma}>t)/M_{\Gamma}(x)$ and its dual counterpart $\hP_y(\tau_{\Gamma}>t)/\hM_{\Gamma}(y)$ for $x,y \in \Gamma$. With the boundary Harnack principle \eqref{eq:BHP}, the Ikeda-Watanabe formula \eqref{eq:IW_density} and Lemma \ref{lem:M} at hand, one can directly repeat the proof of \cite[Lemma 4.2]{BB04} and its extension \cite[Lemma 3.2]{un_BKLP23} to get the following.
\begin{proposition}\label{prop:M_surv}
	\hspace{2em}
	\begin{enumerate}
		\item\label{prop:M_surv_1} For every $R \in (0,\infty)$, there exists a constant $c_1=c_1(\alpha,\Gamma,\lambda,R)$ such that
		\[
		c_1^{-1}M_{\Gamma}(x)t^{-\beta/\alpha} \leq \P_x(\tau_{\Gamma}>t) \leq c_1 M_{\Gamma}(x)t^{-\beta/\alpha}, \quad x \in \Gamma_{Rt^{1/\alpha}},\,t>0.
		\] 
		\item\label{prop:M_surv_2} There exists a constant $c_2=c_2(\alpha,\Gamma,\lambda)$, such that
		\[
		\P_x(\tau_{\Gamma}>t ) \leq c_2 \big( t^{-\beta/\alpha}+t^{-1}|x|^{\alpha-\beta} \big)M_{\Gamma}(x), \quad x \in \Gamma, \,t>0.
		\]
		\item\label{prop:M_surv_3} Statements \eqref{prop:M_surv_1} and \eqref{prop:M_surv_2} hold true for the dual process $\hbfX$, too.
	\end{enumerate}
\end{proposition}
Proposition \ref{prop:M_surv} together with \eqref{eq:rho_f} and \eqref{eq:11} yield two crucial estimates on the conditioned kernel $\rho_t$. First, from \eqref{prop:M_surv_1} we get that
\begin{equation}\label{eq:rho_est}
\rho_1(x,y) \approx (1+|y|)^{d-\alpha} \frac{\hP_y(\tau_{\Gamma}>1)}{\hM_{\Gamma}(y)}, \quad x \in \Gamma_R, \,y \in \Gamma,
\end{equation}
with the implied constant dependent at most on $\alpha$, $\Gamma$, $\lambda$ and $R$. By parts \eqref{prop:M_surv_2} and \eqref{prop:M_surv_3}, one also concludes that there is a constant $c=c(\alpha,\Gamma,\lambda,R)$, such that
\begin{equation}\label{eq:rho_est2}
\rho_1(x,y) \leq c(1+|y|)^{-d-\hbeta}, \quad x \in \Gamma_R,\,y \in \Gamma.
\end{equation}
In a similar way, one also obtains
\begin{equation}\label{eq:rho_est3}
	\rho_1(x,y) \leq c(1+|x|)^{-d-\beta}, \quad y \in \Gamma_R,\,x \in \Gamma.
\end{equation}
\begin{remark}\label{rem:3}
We note that it follows from \eqref{eq:rho_density} and \eqref{eq:rho_est} that the function
\[
\Gamma \ni y \mapsto (1+|y|)^{d-\alpha} \frac{\hP_y(\tau_{\Gamma}>1)}{\hM_{\Gamma}(y)}
\]
is integrable with respect to $M_{\Gamma}(y)\hM_{\Gamma}(y)\ud y$.
\end{remark}
Following \cite{BJKP23_jfa} and \cite{un_BKLP23}, we now apply a change of time and rescaling of space by setting
\begin{equation}\label{eq:26}
\ell_t(x,y):= \rho_{1-e^{-t}} \big( e^{-t/\alpha}x,y \big), \quad x,y \in \Gamma, \,t>0.
\end{equation}
Using \eqref{eq:rho_density} and \eqref{eq:rho_CH-K} one can quickly verify that
\begin{equation}\label{eq:ell_density}
\int_{\Gamma} \ell_t(x,y)M_{\Gamma}(y)\hM_{\Gamma}(y) \ud y = 1, \quad x \in \Gamma,\,t>0,
\end{equation}
and
\begin{equation}\label{eq:ell_Ch-K}
\int_{\Gamma} \ell_t(x,y)\ell_s(y,z)M_{\Gamma}(y)\hM_{\Gamma}(y)\ud y = \ell_{s+t}(x,z), \quad x,z \in \Gamma,\,s,t>0.
\end{equation}
To wit, $\ell_t$ is a transition probability density on $\Gamma$ with respect to $M_{\Gamma}(y)\hM_{\Gamma}(y)\ud y$. The corresponding (Ornstein-Uhlenbeck type) semigroup is defined as follows:
\begin{equation}\label{eq:Lt}
L_t f(y) = \int_{\Gamma} \ell_t(x,y)f(x)M_{\Gamma}(x)\hM_{\Gamma}(x)\ud x, \quad y \in \Gamma, \,t>0.
\end{equation}
It is easy to see that the operators $L_t$ are bounded and linear, thus continuous on $L^1(M_{\Gamma}(y)\hM_{\Gamma}(y)\ud y)$. Moreover, denote a non-negative function $\vphi$ satisfying $\int_{\Gamma} \vphi(x)M_{\Gamma}(x)\hM_{\Gamma}(x)\ud x=1$ a \emph{density}. The Fubini-Tonelli theorem entails that for every $f \geq 0$,
\[
\int_{\Gamma} L_tf(y) M_{\Gamma}(y)\hM_{\Gamma}(y)\ud y = \int_{\Gamma} f(x) M_{\Gamma}(x)\hM_{\Gamma}(x) \ud x.
\]
Thus, the operators $L_t$ preserve densities.

The following result is crucial for our development. As we shall see, the function $\vphi$ from Theorem 
\ref{thm:s_density} will provide a density of the Yaglom measure, see Theorem \ref{thm:Y2}.
\begin{theorem}\label{thm:s_density}
	Let $\Gamma$ be a $\kappa$-fat cone. There is a unique stationary density $\vphi$ for the operators $L_t$, $t>0$.
\end{theorem}
The proof of Theorem \ref{thm:s_density} follows the one of \cite[Theorem 3.4]{un_BKLP23} with the necessary adaptations resulting from a (slightly) different reference measure and the lack of symmetry of the kernels $\rho_t$. For this reason, we provide the key calculations for the convenience of the reader.
\begin{proof}[Proof of Theorem \ref{thm:s_density}]
	Fix $t>0$. First observe that by \eqref{eq:rho_scaling2}, for a non-negative function $f$,
	\begin{align*}
			L_tf(z) &= \int_{\Gamma} \rho_{1-e^{-t}} \big( e^{-t/\alpha}y,z \big) f(y)M_{\Gamma}(y)\hM_{\Gamma}(y)\ud y \\ &= \int_{\Gamma} e^{t(d+\beta+\hbeta)/\alpha} \rho_{e^t-1} \big( y,e^{t/\alpha} z \big) f(y) M_{\Gamma}(y)\hM_{\Gamma}(y) \ud y.
	\end{align*}
	Consider a family $F$ of non-negative functions $f$ of the form
	\[
	f(y) = \int_{\Gamma_1} \rho_1(x,y) \,\eta(\od x), \quad y \in \Gamma.
	\]
	where $\eta$ is some sub-probability measure concentrated on $\Gamma_1$. We first observe that by \eqref{eq:rho_est}, for every $f \in F$,
	\[
	f(y) \lesssim (1+|y|)^{-d-\alpha} \frac{\hP_y(\tau_{\Gamma}>1)}{\hM_{\Gamma}(y)}, \quad y \in \Gamma,
	\] 
	with the implied constant dependent at most on $\alpha,\Gamma$ and $\lambda$. Thus, by Remark \ref{rem:3}, the family $F$ is uniformly integrable with respect to $M_{\Gamma}(y)\hM_{\Gamma}(y)\ud y$. Moreover, by the Fubini-Tonelli theorem, \eqref{eq:rho_CH-K} and \eqref{eq:rho_scaling2}, for $f \in F$,
	\begin{align*}\begin{aligned}
		L_t f(z) &= \int_{\Gamma_1} e^{t(d+\beta+\hbeta)/\alpha} \int_{\Gamma} \rho_1(x,y) \rho_{e^t-1}\big(y,e^{t/\alpha}z\big) M_{\Gamma}(y)\hM_{\Gamma}(y)\ud y\,\eta(\od z) \\ &= \int_{\Gamma_1} e^{t(d+\beta+\hbeta)/\alpha} \rho_{e^t} \big( x,e^{t/\alpha} z\big)\,\eta(\od z) \\ &= \int_{\Gamma_1} \rho_1 \big( e^{-t/\alpha}x,z \big)\,\eta(\od z) \\ &= \int_{\Gamma_1} \rho_1(x,z)\widetilde{\eta}(\od z)\end{aligned}
	\end{align*}
	for some sub-probability measure $\widetilde{\eta}$ concentrated on $\Gamma_1$. It follows that $L_t F \subseteq F$. Since $F$ is clearly a convex set, one can follow steps in \cite[proof of Theorem 3.4]{un_BKLP23} and conclude that there is a density $\vphi \in \overline{F}$ satisfying $L_t \vphi = \vphi$. The uniqueness of $\vphi$ follows then from the strict positivity of $\ell_t$ and the fact that the operators $L_t$ commute, see \cite[proof of Theorem 3.2]{BJKP23_jfa}. 
\end{proof}
We now concentrate of the stationary density $\vphi$ and its properties. First, by Kulik and Scheutzov \cite[Theorem 1 and Remark 2]{KulikScheutzow15}, it is the limit of transition probabilities $\ell_t$ in the following sense: for every $x \in \Gamma$,
\begin{equation}\label{eq:25}
\lim_{t \to \infty} \int_{\Gamma} |\ell_t(x,y)-\vphi(y)| M_{\Gamma}(y)\hM_{\Gamma}(y)\ud y = 0.
\end{equation}
Taking \eqref{eq:26} into account, we intend to re-phrase the limit \eqref{eq:25} by replacing $\ell_t(x,y)$ with $\rho_1(x,y)$ and  $t \to \infty$ with $\Gamma \ni x \to 0$. First we observe that the convergence \eqref{eq:25} is in fact uniform in $x$ in every bounded subset $A \subseteq \Gamma$. Indeed, for some fixed $x_0 \in A$, by \eqref{eq:ell_Ch-K}, \eqref{eq:26} and \eqref{eq:rho_est}, we have
\begin{align}\label{eq:27}\begin{aligned}
	&\int_{\Gamma} |\ell_{t+1}(x,y)-\vphi(y)|M_{\Gamma}(y)\hM_{\Gamma}(y)\ud y \\ = &\int_{\Gamma} \bigg| \int_{\Gamma} \ell_1(x,z) \big( \ell_t(z,y)-\vphi(y) \big) M_{\Gamma}(z)\hM_{\Gamma}(z)\ud z\bigg| M_{\Gamma}(y)\hM_{\Gamma}(y)\ud y \\ \leq &\int_{\Gamma} c\ell_1(x_0,z) \int_{\Gamma} |\ell_t(z,y)-\vphi(y)| M_{\Gamma}(y)\hM_{\Gamma}(y)\ud y \, M_{\Gamma}(z)\hM_{\Gamma}(z)\ud z.\end{aligned}
\end{align}
By \eqref{eq:25}, for every $z \in \Gamma$, the inner integral converges to $0$ as $t \to \infty$. Using \eqref{eq:ell_density} and Theorem \ref{thm:s_density}, we also see that it is uniformly bounded by $2$. Thus, by the dominated convergence theorem, the iterated integral \eqref{eq:27} converges to $0$ as $t \to \infty$, and the claim follows. Put differently,
\begin{equation}\label{eq:28}
\lim_{t \to \infty} \int_{\Gamma} \Big|\rho_{1-e^{-t}}  \big(e^{-t/\alpha}x,y\big) - \vphi(y) \Big| M_{\Gamma}(y) \hM_{\Gamma}(y)\ud y = 0
\end{equation}
uniformly in $x \in A$ on every bounded subset $A \subseteq \Gamma$.

We now derive a desired reformulation of \eqref{eq:25}.
\begin{corollary}\label{cor:rho_L1_conv}
	Let $\Gamma$ be a $\kappa$-fat cone. Then
	\[
	\lim_{\Gamma \ni x \to 0} \int_{\Gamma} |\rho_1(x,y)-\vphi(y)|M_{\Gamma}(y)\hM_{\Gamma}(y)\ud y=0.
	\]
\end{corollary}
\begin{proof}
	Let $A$ be a bounded subset of $\Gamma$. The scaling property \eqref{eq:rho_scaling} yields that
	\[
	\rho_{1-e^{-t}}  \big(e^{-t/\alpha}z,u\big) = (1-e^{-t})^{-(d+\beta+\hbeta)/\alpha} \rho_1 \Big( (e^t-1)^{-1/\alpha}z,(1-e^{-t})^{-1/\alpha}u \Big), \quad z,u \in \Gamma.
	\] 
	Thus, \eqref{eq:28} together with \eqref{eq:ell_density} and the triangle inequality entail that
	\[
	\lim_{t \to \infty}\int_{\Gamma} \bigg| \rho_1 \Big( (e^t-1)^{-1/\alpha}z,(1-e^{-t})^{-1/\alpha}u \Big) - \vphi(u) \bigg| M_{\Gamma}(u)\hM_{\Gamma}(u) \ud u=0
	\]
	uniformly in $z \in A$. By the continuity of dilations in $L^1(\Rd)$, the change of variables $y=(1-e^{-t})^{-1/\alpha}u$, \eqref{eq:M_hom} with Remark \ref{rem:4}, and the triangle inequality,
	\[
	\lim_{t \to \infty}\int_{\Gamma} \bigg| \rho_1 \Big( (e^t-1)^{-1/\alpha}z,y \Big) - \vphi(y) \bigg| M_{\Gamma}(y)\hM_{\Gamma}(y) \ud y=0
	\]
	uniformly in $z \in A$. To conclude the proof, we take $A=\Sd$ and $x = (e^t-1)^{-1/\alpha}z$ with $t = \ln (1+|x|^{-\alpha})$ and $z = x/|x| \in A$.
\end{proof}
In fact, the stationary density $\vphi$ may be refined to a continuous function on $\Gamma$.
\begin{lemma}\label{lem:vphi_mod}
	After a modification on a set of Lebesgue measure zero, $\vphi$ is continuous on $\Gamma$ and
	\[
	\vphi(y) \approx (1+|y|)^{-d-\alpha} \frac{\hP_y(\tau_{\Gamma}>1)}{\hM_{\Gamma}(y)}, \quad y \in \Gamma.
	\] 
\end{lemma}
\begin{proof}
	First we note that
	\[
	\vphi(y) \approx (1+|y|)^{-d-\alpha} \frac{\hP_y(\tau_{\Gamma}>1)}{\hM_{\Gamma}(y)} \quad {\rm a.e.}\, y \in \Gamma,
	\] 
	by virtue of Corollary \ref{cor:rho_L1_conv} and \eqref{eq:rho_est}. Moreover, by Theorem \ref{thm:s_density}, $\vphi(y)=L_1 \vphi(y)$ a.e. $y \in \Gamma$. Hence, it is enough to prove that $L_1 \vphi$ is continuous on $\Gamma$.
	
	By \eqref{eq:rho_est} and \eqref{eq:26}, we have
	\[
	\ell_1(x,y) \approx \frac{\P_{e^{-1/\alpha}x}\big(\tau_{\Gamma}>1-e^{-1}\big)}{M_{\Gamma} \big(e^{-1/\alpha}x\big)} p_{1-e^{-1}} \big( e^{-1/\alpha}x,y \big) \frac{\hP_y(\tau_{\Gamma}>1)}{\hM_{\Gamma}(y)}, \quad x,y \in \Gamma.
	\]
	By \eqref{eq:surv_scaling} and Remark \ref{rem:2}, 
	\[
	\P_{e^{-1/\alpha}x}\big(\tau_{\Gamma}>1-e^{-1}\big) \approx \P_x(\tau_{\Gamma}>e-1) \approx \P_x(\tau_{\Gamma}>1).
	\]
	Therefore, if we let $R>1$, then by Proposition \ref{prop:M_surv}, homogeneity of $M_{\Gamma}$ \eqref{eq:M_hom} and \eqref{eq:11} with \eqref{eq:12},
	\[
	\ell_1(x,y) \lesssim (1+|x|)^{-d-\alpha} \frac{\P_x(\tau_{\Gamma}>1)}{M_{\Gamma}(x)}, \quad x \in \Gamma, y \in \Gamma_R.
	\]
	Here the implied constant may depend on $R$. The dominated convergence theorem yields the continuity of $L_1\vphi$ on $\Gamma_R$, and by arbitrary choice of $R$ we conclude the proof.
\end{proof}
We are now able to assert that there is a finite limit of $\rho_t(x,y)$ as $\Gamma \ni x \to 0$, expressed by means of the stationary density $\vphi$. 
\begin{theorem}\label{thm:rho_conv}
	Let $\Gamma$ be a $\kappa$-fat cone. For every $t>0$, uniformly in $y \in \Gamma$,
	\[
	\lim_{\Gamma \ni x \to 0} \rho_t(x,y) = t^{-(d+\beta+\hbeta)/\alpha}\vphi \big( t^{-1/\alpha}y \big).
	\]
\end{theorem}
\begin{remark}\label{rem:6}
	By Remark \ref{rem:5}\eqref{rem:5_2}, $\beta = 0$ if and only if $\hbeta=0$, and then one has $M_{\Gamma} = \hM_{\Gamma} =\ind_{\Gamma}$. Consequently, $\rho_t(x,y) = p_t(x,y)$ and using the Chapman-Kolmogorov property \eqref{eq:dhk_CH-K} together with the scalings \eqref{eq:dhk_scaling2} one can verify that $\vphi(y) = p_1(0,y)$ is the stationary density from Theorem \ref{thm:s_density}. Therefore, the statement of Theorem \ref{thm:rho_conv} trivialises to the continuity property of the heat kernel.
\end{remark}
\begin{proof}[Proof of Theorem \ref{thm:rho_conv}]
	By Remark \ref{rem:6}, we may assume that $\beta>0$ and $\hbeta>0$. We first claim that there is a constant $c \in (0,\infty)$ dependent at most on $\alpha$, $d$, $\Gamma$ and $\lambda$, such that for all $x \in \Gamma_1$ and $y \in \Gamma$,
	\begin{equation}\label{eq:29}
	\int_{\Gamma} \big| \rho_1 \big( 2^{1/\alpha}x,z \big) - \vphi(z) \big| \rho_1 \big( z,2^{1/\alpha}y \big)M_{\Gamma}(z)\hM_{\Gamma}(z)\ud z \leq c(1+|y|)^{-\hbeta}.
	\end{equation}
	To this end, denote $\tilde{y} = 2^{1/\alpha}y$, so that $|\tilde{y}| \approx |y|$. By \eqref{eq:rho_est}, Lemma \ref{lem:vphi_mod}, \eqref{eq:rho_f} and Proposition \ref{prop:M_surv}, for all $z \in \Gamma$ we have
	\[
	\big| \rho_1 \big( 2^{1/\alpha}x,z \big) - \vphi(z) \big| \rho_1 \big( z,\tilde{y} \big)M_{\Gamma}(z)\hM_{\Gamma}(z) \lesssim (1+|z|)^{-d-\alpha} (1+|z-\tilde{y}|)^{-d-\alpha} (1+|y|)^{\alpha-\hbeta}.
	\]
	We split the integral \eqref{eq:29} as follows. For $z \in A:=B(\tilde{y},|\tilde{y}|/2)$ we have $|z| \approx |\tilde{y}|\approx |y|$ and $1+|z-\tilde{y}| \geq 1$, which implies
	\[
	\int_A \big| \rho_1 \big( 2^{1/\alpha}x,z \big) - \vphi(z) \big| \rho_1 \big( z,2^{1/\alpha}y \big)M_{\Gamma}(z)\hM_{\Gamma}(z)\ud z \lesssim |y|^d(1+|y|)^{-d-\hbeta} \leq (1+|y|)^{-\hbeta}.
	\]
	On $\Gamma \setminus A$ we use the fact that $1+|z| \geq 1$ to get
	\begin{align*}
		\int_{\Gamma \setminus A} \big| \rho_1 \big( 2^{1/\alpha}x,z \big) - \vphi(z) \big| \rho_1 \big( z,2^{1/\alpha}y \big)M_{\Gamma}(z)\hM_{\Gamma}(z)\ud z &\lesssim (1+|y|)^{\alpha -\hbeta} \int_{\Gamma \setminus A} (1+|z-\tilde{y}|)^{-d-\alpha} \ud z \\ &\lesssim (1+|y|)^{\alpha -\hbeta} \int_{|\tilde{y}|/2}^{\infty} (1+r)^{-1-\alpha}\ud r \\ &\lesssim (1+|y|)^{-\hbeta}.
	\end{align*}
	Thus, we obtain \eqref{eq:29} as claimed.
	
	Next, we prove that, uniformly in $y \in \Gamma$,
	\begin{equation}\label{eq:32}
	\lim_{\Gamma \ni x \to 0} \int_{\Gamma} \rho_1 \big( 2^{1/\alpha}x,z \big) \rho_1 \big( z,2^{1/\alpha}y \big) M_{\Gamma}(z)\hM_{\Gamma}(z)\ud z = \int_{\Gamma} \vphi(z) \rho_1 \big( z,2^{1/\alpha}y \big) M_{\Gamma}(z)\hM_{\Gamma}(z)\ud z.
	\end{equation}
	Indeed, let $\epsilon>0$. Since $\hbeta>0$, by \eqref{eq:29}, there is $R>0$ depending at most on $\alpha,d,\Gamma,\lambda$ and $\epsilon$ such that
	\begin{equation}\label{eq:31}
	\int_{\Gamma} \big| \rho_1 \big( 2^{1/\alpha}x,z \big) - \vphi(z) \big| \rho_1 \big( z,2^{1/\alpha}y \big)M_{\Gamma}(z)\hM_{\Gamma}(z)\ud z < \epsilon,
	\end{equation}
	if only $y \in \Gamma \setminus \Gamma_R$. If $y \in \Gamma_R$, then by \eqref{eq:rho_est3} and Corollary \ref{cor:rho_L1_conv},
	\begin{align}\label{eq:30}\begin{aligned}
	&\int_{\Gamma} \big| \rho_1 \big( 2^{1/\alpha}x,z \big) - \vphi(z) \big| \rho_1 \big( z,2^{1/\alpha}y \big)M_{\Gamma}(z)\hM_{\Gamma}(z)\ud z \\ \lesssim &\int_{\Gamma} \big| \rho_1 \big( 2^{1/\alpha}x,z \big) - \vphi(z) \big| M_{\Gamma}(z)\hM_{\Gamma}(z)\ud z \\ < &\epsilon,\end{aligned}
	\end{align}
	for all $y \in \Gamma_R$ and $x \in \Gamma_1$ small enough. Note that the implied constant in \eqref{eq:30} does not depend on $y$. Putting together \eqref{eq:31} with \eqref{eq:30}, we conclude \eqref{eq:32}, as desired. 
	
	For the final part we note that by the scalings \eqref{eq:rho_scaling2} and the Chapman-Kolmogorov property \eqref{eq:rho_CH-K},
	\begin{align*}
		\rho_1(x,y) &= 2^{(d+\beta+\beta)/\alpha} \rho_2 \big( 2^{1/\alpha}x,2^{1/\alpha}y \big) \\ &= 2^{(d+\beta+\beta)/\alpha} \int_{\Gamma} \rho_1 \big( 2^{1/\alpha}x,z \big) \rho_1 \big( z,2^{1/\alpha}y \big) M_{\Gamma}(z)\hM_{\Gamma}(z)\ud z.
	\end{align*}
	It follows now from \eqref{eq:32}, \eqref{eq:rho_scaling2} and Theorem \ref{thm:s_density} that
	\begin{align*}
		\lim_{\Gamma \ni x \to 0}\rho_1(x,y) &= 2^{(d+\beta+\hbeta)/\alpha} \int_{\Gamma} \vphi(z)\rho_1 \big( z,2^{1/\alpha}y \big) M_{\Gamma}(z)\hM_{\Gamma}(z)\ud z \\ &= \int_{\Gamma} \vphi(z)\rho_{1/2} \big( 2^{-1/\alpha}z,y \big) M_{\Gamma}(z)\hM_{\Gamma}(z)\ud z \\ &= L_{\ln 2} \vphi(y) \\ &= \vphi(y).
	\end{align*}
\end{proof}
\begin{remark}
	Recall that for every $x \in \Gamma$ and $t > 0$,
	\[
	\int_{\Gamma} \rho_t(x,y)M_{\Gamma}(y)\hM_{\Gamma}(y)\ud y = 1.
	\]
	By \eqref{eq:rho_est}, Theorem \ref{thm:rho_conv} and the dominated convergence  theorem, we also have
	\[
	\int_{\Gamma} \vphi(y)M_{\Gamma}(y)\hM_{\Gamma}(y)\ud y = 1.
	\]
\end{remark}
With Theorem \ref{thm:rho_conv} at hand, the following corollary readily follows. Recall that the comparability of the survival probability and the Martin kernel was given in Proposition \ref{prop:M_surv}. Now we focus on the limiting behaviour as $\Gamma \ni x \to 0$ and derive the aforementioned asymptotic relation.
\begin{corollary}\label{cor:surv_M_as}
	Let $\Gamma$ be a $\kappa$-fat cone. Then, for every $t>0$,
	\[
	\lim_{\Gamma \ni x \to 0} \frac{\P_x (\tau_{\Gamma}>t)}{M_{\Gamma}(x)} = C_1 t^{-\beta/\alpha}, \qquad \text{where} \qquad C_1 = \int_{\Gamma} \vphi(z)\hM_{\Gamma}(z)\ud z \in (0,\infty).
	\]
\end{corollary}
\begin{proof}
	First let $t=1$. By \eqref{eq:surv_prob}, 
	\[
	\frac{\P_x (\tau_{\Gamma}>1)}{M_{\Gamma}(x)} = \int_{\Gamma} \frac{p_1^{\Gamma}(x,y)}{M_{\Gamma}(x)}\ud y = \int_{\Gamma} \rho_1(x,y)\hM_{\Gamma}(y)\ud y.
	\]
	We let $\Gamma \ni x \to 0$ and use \eqref{eq:rho_est}, the dominated convergence theorem and Theorem \ref{thm:rho_conv} to get the claim. For the general case we use the scalings \eqref{eq:surv_scaling} and \eqref{eq:Mk_hom}.
\end{proof}
We are now ready to prove our second main result, c.f. Theorem \ref{thm:Yaglom}. Note that the limiting Yaglom measure is absolutely continuous with respect to the Lebesgue measure, with the density given by objects pertaining to both $\bfX$ and $\hbfX$.
\begin{theorem}\label{thm:Y2}
	Let $\bfX$ be a strictly stable L\'{e}vy process satisfying {\bf A}. Assume $\Gamma$ is a $\kappa$-fat cone. There is a probability measure $\mu$ concentrated on $\Gamma$ such that for every Borel set $A \subseteq \Gamma$ and every probability measure $\gamma$ on $\Gamma$ satisfying $\int_{\Gamma}(1+|y|)^{\alpha}\,\gamma(\od y)<\infty$,
	\begin{equation*}
		\lim_{t \to \infty} \P_{\gamma} \big( t^{-1/\alpha}X_t \in A \big| \tau_{\Gamma}>t \big) = \mu(A), \quad x \in \Gamma,
	\end{equation*}
	where
	\[
	\mu(A) = \frac{1}{C_1} \int_A \vphi(y)\hM_{\Gamma}(y)\ud y, \quad A \subseteq \Gamma.
	\]
\end{theorem}
\begin{proof}
	Assume first that the process starts from a fixed point $x \in \Gamma$ with probability $1$, that is the initial distribution $\gamma$ is a Dirac delta at $x$. By \eqref{eq:surv_prob} and the scaling property \eqref{eq:dhk_scaling},
	\begin{align*}
		\P_x \big( t^{-1/\alpha}X_t \in A \,\big|\, \tau_{\Gamma}>t \big) &= \frac{\P_x \big( t^{-1/\alpha}X_t \in A, \, \tau_{\Gamma}>t \big)}{\P_x (\tau_{\Gamma}>t)} \\ &= \frac{\P_{t^{-1/\alpha}x} \big( X_1 \in A, \, \tau_{\Gamma}>1 \big)}{\P_{t^{-1/\alpha}x} (\tau_{\Gamma}>1)} \\ &= \int_A \frac{p_1^{\Gamma}\big( t^{-1/\alpha}x,y \big)}{M_{\Gamma}(t^{-1/\alpha}x)}\ud y \cdot \frac{M_{\Gamma}(t^{-1/\alpha}x)}{\P_{t^{-1/\alpha}x}(\tau_{\Gamma}>1)} \\ &= \int_A \rho_1 (t^{-1/\alpha}x,y)\hM_{\Gamma}(y)\ud y \cdot \frac{M_{\Gamma}(t^{-1/\alpha}x)}{\P_{t^{-1/\alpha}x}(\tau_{\Gamma}>1)}.
	\end{align*}
	Thus, using Corollary \ref{cor:surv_M_as}, Theorem \ref{thm:rho_conv}, \eqref{eq:rho_est} and the dominated convergence theorem we get the conclusion in this case. Note that, in fact, the convergence is uniform in $x \in B$ for every bounded $B \subseteq \Gamma$.
	
	For the general initial distribution $\gamma$ satisfying the integrability condition $\int_{\Gamma} (1+|y|)^{\alpha}\,\gamma(\od y)<\infty$, one can follow the proof of \cite[Theorem 3.12]{un_BKLP23} line by line to deduce the desired result. For this reason, we omit the details.
	
\end{proof}
\begin{remark}
	 As a complement of Remark \ref{rem:4} and Proposition \ref{prop:M_surv}\eqref{prop:M_surv_3}, we note that in exactly the same way one may obtain the dual counterparts of Corollary \ref{cor:surv_M_as} and Theorem \ref{thm:Yaglom}. To wit, one can show that there is a stationary density $\widehat{\vphi}$ for the Ornstein-Uhlenbeck operators $\widehat{L}_t$ given by the integral kernels $\widehat{\ell}_t$ associated to the dual conditioned kernel $\widehat{\rho}_t(x,y):=\rho_t(y,x)$ for $x,y \in \Gamma$. Accordingly, c.f. Lemma \ref{lem:vphi_mod},
	 \[
	 \widehat{\vphi}(y) \approx (1+|y|)^{-d-\alpha} \frac{\P_y(\tau_{\Gamma}>t)}{M_{\Gamma}(y)}, \quad y \in \Gamma,
	 \]
	 and as a consequence, 
	 \[
	 \lim_{\Gamma \ni x \to 0} \frac{\hP_x (\tau_{\Gamma}>t)}{\hM_{\Gamma}(x)} = \widehat{C}_1 t^{-\hbeta/\alpha}, \qquad \text{where} \qquad \widehat{C}_1 = \int_{\Gamma} \widehat{\vphi}(z)M_{\Gamma}(z)\ud z \in (0,\infty),
	 \]
	 and for every $\kappa$-fat cone $\Gamma$ and every probability measure $\gamma$ satisfying $\int_{\Gamma}(1+|y|)^\alpha\,\gamma(\od y)<\infty$,
	 	\[
	 \lim_{t \to \infty} \hP_{\gamma} \big( t^{-1/\alpha}\widehat{X}_t \in A \big| \tau_{\Gamma}>t \big) = \widehat{\mu}(A), \quad A \subseteq \Gamma,
	 \]
	 where
	 \[
	 \widehat{\mu}(A) := \frac{1}{\widehat{C}_1} \int_A \widehat{\vphi}(y)M_{\Gamma}(y)\ud y, \quad A \subseteq \Gamma.
	 \]
	 The necessary changes essentially consist of replacing the underlying objects with its dual counterparts. The details are left for the reader.
\end{remark}
\bibliographystyle{abbrv}
\bibliography{bib-file}
\end{document}